\documentclass[12pt]{article}
\usepackage{amsmath,amssymb,amstext,dsfont,fancyvrb,float,fontenc,graphicx,subfigure,theorem,hyperref,tikz,blkarray,mathrsfs,multirow}
\usepackage[utf8]{inputenc}
\usepackage{kbordermatrix}
\usepackage{comment}

\usepackage{tikz,everypage}
\usepackage[letterpaper]{geometry}
\newfont{\footsc}{cmcsc10 at 8truept}
\newfont{\footbf}{cmbx10 at 8truept}
\newfont{\footrm}{cmr10 at 10truept}

\usepackage{fancyhdr}
\usepackage{relsize}
\usepackage{sectsty}
\allsectionsfont{\larger[-1]} 

\renewcommand\paragraph{\@startsection{paragraph}{4}{\z@}
                                    {2ex \@plus.5ex \@minus.2ex}
                                    {-1em}
                                    {\normalfont\normalsize\bfseries}}

\renewcommand\subparagraph{\@startsection{subparagraph}{5}{\parindent}
                                       {2ex \@plus.5ex \@minus .2ex}
                                       {-1em}
                                      {\normalfont\normalsize\bfseries}}

\newlength{\BiblioSpacing}
\renewenvironment{thebibliography}[1]{
\begin{oldthebibliography}{#1}
\setlength{\parskip}{\BiblioSpacing}
\setlength{\itemsep}{\BiblioSpacing}
}
{
\end{oldthebibliography}
}

\usepackage[strict]{changepage}
\def\abstractname{Abstract -}   % <-----------------
\def\abstract{\begin{adjustwidth}{1cm}{1cm} \par    \footnotesize \noindent {\bf \abstractname} 
\def\endabstract{ \end{adjustwidth} \smallskip }}
{\theorembodyfont{\itshape}\newtheorem{theorem}{Theorem}[section]}
{\theorembodyfont{\itshape}\newtheorem{proposition}[theorem]{Proposition}}
{\theorembodyfont{\itshape}}
{\theorembodyfont{\itshape}\newtheorem{lemma}[theorem]{Lemma}}
{\theorembodyfont{\itshape}\newtheorem{corollary}[theorem]{Corollary}}
{\theorembodyfont{\rm}}
{\theorembodyfont{\rm}}
{\theorembodyfont{\rm}}
{\theorembodyfont{\rm }}

\title{\Large\bf A combinatorial approach to counting primitive periodic and primitive pseudo orbits on circulant graphs}
\author{\sc L. Engelthaler, I. Hellerman, and T. Hudgins
\thanks{This work was supported by a research grant from The Nancy Cain and Jeffrey A. Marcus Science Endowment in Honor of President Donald A. Cowan.}}
\begin{document}
\setcounter{page}{1}
%\date{}
%%%%%%%%%%%%%%%%%%%%%%%%%%%%%%%%%%%%%%%%%
\maketitle
\thispagestyle{fancy}

\vskip 1.5em

\begin{abstract}
For families of 4-regular directed circulant graphs with $n$ vertices, we count the number of primitive periodic orbits of length up to at least $n$.
The relevant counting techniques are then extended to count the number of primitive pseudo orbits (sets of distinct primitive periodic orbits) of length up to at least $n$ that lack self-intersections, or that self-intersect only at individual vertices repeated exactly twice (2-encounters of length zero),
%that never intersect at more than a single vertex at a time repeated exactly twice for each self-intersection (2-encounters of length zero), 
for two particular families of 4-regular directed circulant graphs.
We then regard these two families of graphs as families of quantum graphs and use the counting results to compute the variance of the coefficients of the quantum graph's characteristic polynomial.
\end{abstract}


\begin{MSC}
05; 81
\end{MSC}

\section{Introduction}

Quantum chaos is the field that studies the relationships between classically chaotic dynamical systems and quantum systems \cite{G92}.
Kottos and Smilansky proposed quantum graphs as a model of quantum chaos in \cite{KS99}, which involves considering a differential operator or a scattering problem on a graph that has classically chaotic properties like ergodicity or mixing.
One of the statistics related to the quantum spectrum of a quantum graph that Kottos and Smilansky deal with is the variance of the coefficients of a quantum graph's characteristic polynomial, $\langle |a_n|^2 \rangle$.  

Band, Harrison, and Joyner derived a formula for $\langle |a_n|^2 \rangle$ in \cite{BHJ12}, which relies on finite sums over pseudo orbits (collections of periodic orbits). 
(Their approach has some similarities to the formula for the spectral determinant of quantum graphs derived by Akkermans et al. in \cite{Aetal00}.)
The diagonal contribution to the sum for $\langle |a_n|^2 \rangle$ has been computed for families of directed $q$-nary graphs \cite{BHS19}.  
Most recently, the variance formula has been evaluated, in the case of 4-regular directed graphs with incoming and outgoing degree of two and a Discrete Fourier Transform vertex scattering matrix, to a sum that relies only on knowing the numbers of primitive pseudo orbits lacking self-intersections and those that self-intersect only at individual vertices repeated exactly twice (2-encounters of length zero) \cite{HH20, HH21}.  
This is the first spectral statistic in the quantum chaos literature that the authors are aware of that can be fully evaluated from a sum over periodic orbits, without making use of heuristic arguments to justify the neglect of certain orbits from the sum.  
Thus, it becomes useful to investigate counting pseudo orbits of the specified types.  

Circulant graphs have been studied extensively since at least the 1960's \cite{ET69}.  
Properties such as the number of spanning trees \cite{AYI10} or the diameter of a random circulant graph \cite{MS13} have been examined in much detail.
They have also had some spectral properties studied when regarded as quantum graphs \cite{HS19}.
Here we relay some general results regarding the counting of primitive periodic orbits and primitive pseudo orbits of various types on several families of 4-regular directed circulant graphs.
In particular, the main result in section \ref{sec:generalresult} counts the number of primitive periodic orbits of any positive length on a 4-regular directed circulant graph, with incoming and outgoing degree each two, that lacks loops and multiple edges.

The paper is organized in the following way.  Section \ref{sec:background} lays out the relevant terminology, notation, and lemmas for graphs, directed circulant graphs, and periodic and pseudo orbits. 
In section \ref{sec:PO}, some results on counting primitive periodic orbits are given for two specific families of directed circulant graphs, where there is always a directed edge from vertex $i$ to vertex $i+1$, and an additional outgoing edge at vertex $i$ determined by the family.
Then some similar proof techniques are used to prove the main result, which counts primitive periodic orbits of positive length for 4-regular directed circulant graphs, with incoming and outgoing degree each two, that lack loops and multiple edges.
Section \ref{sec:PsO} returns to the same two specific families of directed circulant graphs,  and proves some primitive pseudo orbit counting results, classified by self-intersections (lacking them or having only 2-encounters of length zero).
We regard these families of graphs as families of quantum graphs in section \ref{sec:quantum}, and use the counting results to evaluate the variance of the quantum graph's characteristic polynomial.  
The variance is also considered numerically and is shown to agree with our orbit counting results.
Lastly, some conclusions and future directions are outlined briefly in section \ref{sec:conclusions}.

\section{Graph Theory and Directed Circulant Graphs}
\label{sec:background}

A \textit{graph} $\mathcal{G}$ consists of a set of \textit{vertices} $\mathscr{V}=\{0,1,\dots,n-1\}$ and a set of \textit{edges} $\mathscr{E}$. 
An edge $e=\{v_0,v_1\}$ is a pair of elements from the vertex set $\mathscr{V}$. 
In the case where $v_0=v_1$, the edge is called a \textit{loop}. 
In this paper, we will ignore graphs with loops.

A \textit{bond} $b$ is a directed edge, that is, an ordered pair $(v_0,v_1)$ of elements from the vertex set $\mathscr{V}$; we call the graph $\mathcal{G}$'s set of bonds $\mathscr{B}$.
A graph with directed edges is a \textit{directed graph} (or \textit{digraph}). 
The terms edge and bond are often used interchangeably, but we use the term bond to refer specifically to directed edges. 
The vertex $v_0$ is called the \textit{origin vertex of} $b$ and the vertex $v_1$ is called the \textit{terminal vertex of} $b$. 
The functions $o,t: \mathscr{B}\rightarrow\mathscr{V}$ indicate that $o(b)=v_0$ and $t(b)=v_1$. 
The bond $b$ is \textit{incoming} at $v_1$ and \textit{outgoing} at $v_0$. 
Figure \ref{fig:graph1} shows a directed graph with 7 vertices. 
In this paper, we will ignore graphs with multiple distinct bonds outgoing at a vertex $v_0 \in \mathscr{V}$ and incoming at $v_1 \in \mathscr{V}$.

\begin{figure}[!htb]
\begin{center}
\begin{tikzpicture}
% vertices
\draw[fill=black] (-0.87,1.8) circle (1.5pt);
\draw[fill=black] (0.87,1.8) circle (1.5pt);
\draw[fill=black] (1.95,0.45) circle (1.5pt);
\draw[fill=black] (1.56,-1.25) circle (1.5pt);
\draw[fill=black] (0,-2) circle (1.5pt);
\draw[fill=black] (-1.56,-1.25) circle (1.5pt);
\draw[fill=black] (-1.95,0.45) circle (1.5pt);
% vertex labels
\node at (-1.07,2) {0};
\node at (1.07,2) {1};
\node at (2.17,0.47) {2};
\node at (1.77,-1.27) {3};
\node at (0,-2.25) {4};
\node at (-1.77,-1.27) {5};
\node at (-2.17,0.47) {6};
% bonds
\draw[-latex] (-0.87,1.8) -- (0.75,1.8);
\draw[-latex] (0.87,1.8) -- (1.85,0.55);
\draw[-latex] (1.95,0.45) -- (1.66,-1.15);
\draw[-latex] (1.56,-1.25) -- (0.1,-1.9);
\draw[-latex] (0,-2) -- (-1.46,-1.35);
\draw[-latex] (-1.56,-1.25) -- (-1.95,0.5);
\draw[-latex] (-1.95,0.45) -- (-0.92,1.67);
\draw[-latex] (-0.87,1.8) -- (1.95,0.45);
\draw[-latex] (0.87,1.8) -- (1.56,-1.25);
\draw[-latex] (1.95,0.45) -- (0,-2);
\draw[-latex] (1.56,-1.25) -- (-1.56,-1.25);
\draw[-latex] (0,-2) -- (-1.95,0.45);
\draw[-latex] (-1.56,-1.25) -- (-0.87,1.8);
\draw[-latex] (-1.95,0.45) -- (0.87,1.8);
\end{tikzpicture}
\caption{\label{fig:graph1} The directed circulant graph $C_7^+(1,2)$}
\end{center}
\end{figure}
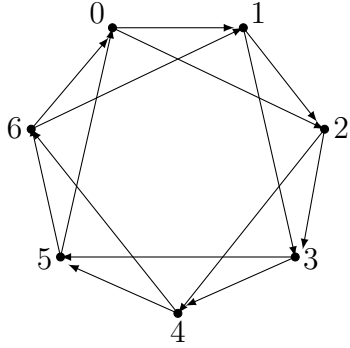

The \textit{incoming degree} $D_v^{in}$ of a vertex $v \in \mathscr{V}$ is the number of incoming bonds at $v$ and the \textit{outgoing degree} $D_v^{out}$ of $v$ is the number of outgoing bonds at $v$. 
The \textit{degree} $D_v$ of $v$ is the sum of the incoming and outgoing degrees of $v$. 
A graph is \textit{$k$-regular} if every vertex $v \in \mathscr{V}$ has degree $D_v=k$. 
In this paper, we are particularly concerned with 4-regular directed graphs such that $D_v^{in}=D_v^{out}=2$ for all $v \in \mathscr{V}$.

The \textit{adjacency matrix} of a graph $\mathcal{G}$ is a square matrix $A$ whose rows and columns are indexed by the vertices, $i,j=0,1,\dots,n-1$, such that $A_{i,j}$ is equal to the number of edges from vertex $i$ to vertex $j$. 
As we are ignoring graphs with loops and multiple bonds with the same direction between any pair of vertices, all entries of our adjacency matrices will have only zeroes and ones as entries, and will have only zeroes along the diagonal.
The adjacency matrix of an undirected graph is symmetric ($A_{i,j}=A_{j,i}$). 
The adjacency matrix of a directed graph need not be symmetric.

A square matrix is \textit{circulant} if $A_{i,j} = A_{i+1, j+1}$, where $n$ is the number of rows or columns in the matrix, and the indices are taken mod $n$.
In other words, $A$ is defined by a row vector, the first row, and the elements of each subsequent row are a cyclic shift one entry to the right.
The adjacency matrix of the directed graph in figure \ref{fig:graph1},
\begin{equation}
A = \kbordermatrix{
& 0 & 1 & 2 & 3 & 4 & 5 & 6 \\
0 & 0 & 1 & 1 & 0 & 0 & 0 & 0 \\
1 & 0 & 0 & 1 & 1 & 0 & 0 & 0 \\
2 & 0 & 0 & 0 & 1 & 1 & 0 & 0 \\
3 & 0 & 0 & 0 & 0 & 1 & 1 & 0 \\
4 & 0 & 0 & 0 & 0 & 0 & 1 & 1 \\
5 & 1 & 0 & 0 & 0 & 0 & 0 & 1 \\
6 & 1 & 1 & 0 & 0 & 0 & 0 & 0
} \ ,
\end{equation}
% We tried using \blkarray here, but the left and right delimiters were not the right height and we couldn't seem to fix it without changing packages.
is circulant and not symmetric.
Note that if a bond $(v_0, v_0+a \mod n)$ is an outgoing bond at $v_0$, then any vertex $v \in \mathscr{V}$ has an outgoing bond $(v, v+a \mod n)$ for $a\in \mathbb{Z}$.

The \textit{arc} $a$ of a bond $(v_0,v_1)$ is defined as $v_1-v_0 \mod n$. 
The terminal vertex of any bond $(v_0,v_1)$ can then be written as $v_1\equiv v_0+a \mod n$. 
If a graph $\mathcal{G}$ has the same set of arcs of outgoing bonds $\mathscr{A}=\{a_1,a_2,\dots,a_{D_v^{out}}\}$ at any vertex $v$, then we call this a \textit{directed circulant graph}, denoted as $C^+_n(a_1,a_2,\dots,a_{D_v^{out}})$, where $n$ is the number of vertices on the graph.
The graph in figure \ref{fig:graph1} is, in fact, the directed circulant graph $C_7^+(1,2)$.
Note that directed circulant graphs have circulant adjacency matrices.
To avoid multiple distinct bonds with the same direction between vertex pairs, as well as loops, $D_v^{out}$ can be at most $n-1$.
Note that it is typical to limit the length of the largest arc, as multiple edges and loops are not usually allowed on circulant graphs; however, allowing an arc to be a multiple of $n$, or repeating values in the arc list, could in general be done.
Moreover, the definition of an \textit{undirected} circulant graph usually only allows the largest arc to be up to $n/2$, but we do allow for two bonds that run in opposite directions between a pair of vertices; if $a$ and $n-a$ are both in the set of arcs $\mathscr{A}$ on a directed circulant graph, then between every two vertices $v_0$ and $v_1$ such that $v_1 - v_0 \mod n \equiv a$, it is also true that $v_0 - v_1 \mod n \equiv n-a$, but the two corresponding bonds have different initial and terminal vertices.

A 4-regular directed circulant graph with $D_v^{in}=D_v^{out}=2$ for any vertex on the graph can thus be denoted $C_n^+(a_1,a_2)$. 
The only nonzero entries in row $i$ of $A$ are 
\begin{equation}
A_{i,i+a_1}=A_{i,i+a_2}=1 \ ,
\end{equation} 
where the indices are taken mod $n$.
If the smaller of the two possible arcs is $a_1$, we can also denote the graph $C_n^+(a_1,a_1+d)$. 
Note that $a_1 > 0$ and $d > 0$, since we are ignoring graphs with loops and with multiple distinct bonds outgoing at a vertex $v_0 \in \mathscr{V}$ and incoming at $v_1 \in \mathscr{V}$. 
Note that if the number of vertices is $n=a_1+d$, the graph would contain loops, and if $n<a_1+d$, there would be a simpler, equivalent way of describing the graph where $n\nless a_1+d$, so for this paper, $n > a_1+d$. 
On a graph $C_n^+(a_1,a_1+d)$, there is a bond connecting a vertex $i$ to vertex $j$ if $j-i \equiv a_1 \mod n$ or $j-i \equiv a_1+d \mod n$. 
Since these are directed graphs, the same is not necessarily true if $i-j \equiv a_1 \mod n$ or $i-j \equiv a_1+d \mod n$. 
This directionality is specified by the $+$ in the notation for a graph $C^+_n(a_1, a_1+d)$.

A vertex $v_0$ is \textit{adjacent} to a vertex $v_1$ if either $(v_0,v_1) \in \mathscr{B}$ or $(v_1,v_0) \in \mathscr{B}$. 
A \textit{walk} $w$ on a directed graph is an ordered set of bonds and vertices, which can be written $w = v_0, b_0, v_1, b_1, v_2, \dots, v_l$, such that $o(b_i)=v_i$ and $t(b_i)=v_{i+1}$ for $0 \leq i \leq l-1$. 
We call $v_0$ the \textit{origin vertex of} $w$ and $v_l$ the \textit{terminal vertex of} $w$. 
So long as there is a single bond between each pair of vertices, we may denote a walk by just the sequence of adjacent vertices; thus, the walk $w=v_0,v_1,\dots,v_l$. 
A graph is \textit{connected} if, for any two vertices $v_0$ and $v_1$ in the vertex set, there is a walk with origin vertex $v_0$ and terminal vertex $v_1$. 
The \textit{length} of a walk is the number of bonds in the walk, including repetitions. The walk $w$ has length $l$. 
A \textit{circuit} $x$ is a walk that starts and ends at the same vertex. 

A \textit{periodic orbit} $\gamma$ is an equivalence class of closed walks that are all rotations of each other. 
For example, $x_1=v_0,v_1,v_2,v_0$ and $x_2=v_1,v_2,v_0,v_1$, where $\{v_0,v_1,v_2\} \subset \mathscr{V}$, belong to the same periodic orbit. 
The \textit{length} of a periodic orbit $\gamma$ is the number of bonds in $\gamma$, i.e., the number of bonds in any circuit that belongs to $\gamma$. 
Note that the bonds are not counted separately for each circuit in $\gamma$, but only once for each time they appear in a given circuit.

A \textit{pseudo orbit} $\bar{\gamma}$ is a collection of periodic orbits on a graph. 
For instance, the pseudo orbit $\bar{\gamma}=\{\gamma_1,\gamma_2\}$ consists of the periodic orbits $\gamma_1=v_0,v_1,v_2,v_0$ and $\gamma_2=v_1,v_3,v_1$, where $\{v_0,v_1,v_2,v_3\} \subset \mathscr{V}$. 
The \textit{length} of a pseudo orbit is the sum of the lengths of the individual periodic orbits in $\bar{\gamma}$. 
Note that the bonds are counted separately for each periodic orbit they appear in. 
The \textit{null pseudo orbit} $\bar{0}$ is composed of zero periodic orbits, contains no bonds, and is said to have length $l=0$.

A circuit is \textit{primitive} if it does not consist of a shorter circuit repeated multiple times. 
Likewise, a periodic orbit is primitive if it does not consist of a shorter periodic orbit repeated multiple times. 
Note that the number of unique circuits that belong to a primitive periodic orbit is equal to the length of the periodic orbit. 
A pseudo orbit is primitive if it consists only of primitive periodic orbits, none of which are repeated in the collection.

A \textit{self-intersection} or \textit{$\ell$-encounter} is a maximal subsequence of vertices $v_0,v_1,\dots,v_h$ (and, if there is more than one vertex, the bonds between them) in a pseudo orbit $\bar{\gamma}$ that is repeated $\ell$ times within the pseudo orbit. 
The repetition could be due to the same subsequence being used once in each of multiple periodic orbits in $\bar{\gamma}$ or due to the same subsequence being used multiple times within a single periodic orbit in $\bar{\gamma}$, or both. 
Its \textit{encounter length} $h$ is the number of bonds in the sequence. 
If the length is zero, then $h=0$ and the encounter consists of a single vertex repeated $\ell$ times. 
For this paper, we are interested specifically in primitive pseudo orbits with no self-intersection and with only $2$-encounters of length zero.

The \textit{walk sum} $\alpha$ of a walk $w$ is defined as the sum of the arcs of the bonds in $w$. 
The walk sum of a periodic orbit is equal to that of any circuit in the periodic orbit, and the walk sum of a pseudo orbit is the sum of the walk sums of the periodic orbits the pseudo orbit is composed of. 
Note that since the terminal vertex of each bond is the sum of the bond's origin vertex and its arc, the terminal vertex of a walk of length $l$ will be equivalent to the origin vertex plus the walk sum, mod $n$; that is, $v_l\equiv v_0+\alpha \mod n$.

\begin{lemma}
\label{thm:circuitwalksum}
On a directed graph $\mathcal{G}$ with $n$ vertices, the walk sum of any circuit, periodic orbit, or pseudo orbit is equal to $mn$, where $m \in \mathbb{N}$.
\end{lemma}
\begin{proof}
Let $w$ be a walk of length $l$ on a directed graph $\mathcal{G}$. 
In order for $w$ to be a circuit, the terminal vertex must be the same as the origin vertex $v$. 
Let the arcs of the bonds in $w$ be $a_1,a_2\dots,a_l$. 
The walk sum of $w$ is thus $\alpha = a_1+a_2+\cdots+a_l$. 
We must have $v+\alpha \equiv v \mod n$, and thus $\alpha$ is a multiple of $n$. 
Thus, the walk sum of any circuit is a multiple of $n$. 
Since the walk sum of a periodic orbit is the walk sum of a circuit, and the walk sum of a pseudo orbit is the sum of walk sums of periodic orbits, these must also be multiples of $n$.
\end{proof}
\vspace{1em}

We say that a vertex $v$ is \textit{passed} by a circuit $x$ each time that a bond $(v_0,v_1)$ appears in $x$ such that $v-v_0 \mod n < v_1-v_0 \mod n$ (equivalently, the arc of $(v_0,v_1)$ is greater than $v-v_0 \mod n$, the arc required to have a bond $(v_0, v_1)$). 
If $v_0=v$, then $v$ appears in the circuit at this point, and $(v_0,v_1)$ is the outgoing bond from $v$. 
When $v$ is passed by a bond $b$, we will also sometimes say that $b$ passes $v$.
A vertex $v$ is passed by a periodic orbit each time any given circuit in the periodic orbit passes $v$. 
A vertex $v$ is passed by a pseudo orbit $\bar{\gamma}$ each time a periodic orbit in $\bar{\gamma}$ passes $v$.

\begin{lemma}
\label{thm:circuitpassing}
On a directed graph $\mathcal{G}$ with no loops, any circuit, and thus any periodic orbit, with a walk sum equal to $mn$ (where $m \in \mathbb{N}$) will pass a vertex of $\mathcal{G}$ exactly $m$ times. 
Likewise, any pseudo orbit on $\mathcal{G}$ with a walk sum equal to $mn$ will pass a vertex of $\mathcal{G}$ exactly $m$ times. 
\end{lemma}

\begin{proof}
Let $x$ be a circuit of length $l$ on a directed graph $\mathcal{G}$ with a walk sum of $mn$, and let $v$ be a vertex of $\mathcal{G}$. 
Label the vertices of $x$ by $v_0,v_1,\dots, v_l$ and the arcs of the bonds in $x$ by $a_0,a_1,\dots, a_{l-1}$. 
Because $\mathcal{G}$ is a circulant graph with no loops, the arc $a$ of any bond will satisfy $0<a<n$. 
Consider the sum $s_l=v_0+a_0+\cdots+a_{l-1}$, which is equivalent to $v_l \mod n$, and all the partial sums $s_i=v_0+a_0+\cdots+a_{i-1}$, which are equivalent to $v_i \mod n$ for each $0\leq i<l$. 
The lowest sum will be $s_0=v_0$; the highest will be $s_l=v_0+mn$ by Lemma \ref{thm:circuitwalksum}; and all the others will be between, with $s_{i+1}>s_i$ for all $i$, since $a_i>0$ for all $i$. 
Consider any bond in $x$, $(v_i,v_{i+1})$, with arc $a_i$. 
In order for $(v_i,v_{i+1})$ to pass $v$, we must have that $a_i$ is greater than $v-v_i \mod n$; call this number $a'$. 
Then $v \equiv v_i+a' \mod n \equiv v_0+a_0+\cdots+a_{i-1}+a' \mod n$; call this sum $s'$. 
Since $0\leq a'<a_i$, there is a number $s' \equiv v \mod n$ such that $s_i\leq s'<s_{i+1}$. 
Conversely, for any number $s' \equiv v \mod n$ such that $s_0\leq s'<s_l$, there must be exactly one $i$ such that $s_i\leq s'<s_{i+1}$; then $a_i > s'-s_i \equiv v-v_i \mod n$, and so the bond $(v_i,v_{i+1})$, which has arc $a_i$, passes $v$. 
Thus, $x$ passes $v$ exactly once for each number $s' \equiv v \mod n$ such that $s_0\leq s'<s_l$, and since $s_l-s_0=mn$, there are exactly $m$ of these numbers. 
Therefore $x$ passes $v$ exactly $m$ times.

Since the number of times a pseudo orbit passes a vertex is simply the sum of the times it occurs in the periodic orbits that make up the pseudo orbit, we can also say that a pseudo orbit passes $v$ $m$ times if its walk sum is $mn$.
\end{proof}
\vspace{1em}

If $a_1=1$, then a circuit $x$ passes a vertex $v$ each time a bond of arc $d+1$ with origin vertex $v_0\in\{v-d,v-d+1,\dots,v-1\}$ appears in $x$ as well as each time $v$ appears in $x$.

In the following lemma, we will use the functions $o$ and $t$ on a set of bonds, rather than an individual bond: if $\mathcal{B}=\{b_1,b_2,\dots,b_l\}$, then $o(\mathcal{B})$ is defined as the multiset (meaning a set that allows repetition, but is still not ordered) of vertices $\{o(b_1),o(b_2),\dots,o(b_l)\}$, and $t(\mathcal{B})$ is defined analogously.

\begin{lemma}
\label{thm:bondstoPsO}
On a directed graph $\mathcal{G}$, for any nonempty set $\mathcal{B}$ of distinct bonds of $\mathcal{G}$ such that $o(\mathcal{B})=t(\mathcal{B})$ and $o(\mathcal{B})$ has no vertex repeated more than twice, the number of distinct pseudo orbits containing exactly the bonds of $\mathcal{B}$ is $2^N$, where $N$ is the number of repeated vertices in $o(\mathcal{B})$.
\end{lemma}

\begin{proof}
Let $\mathcal{G}$ be a directed graph, and let $\mathcal{B}$ be a nonempty set of distinct bonds of $\mathcal{G}$ such that $o(\mathcal{B})=t(\mathcal{B})$, no vertex of $o(\mathcal{B})$ is repeated more than twice, and there are exactly $N$ vertices repeated twice. 
Let $\bar{\gamma}$ be a pseudo orbit containing exactly the bonds of $\mathcal{B}$; then $o(\mathcal{B})$ is the set of vertices used in $\bar{\gamma}$. 
Let $b_0=(v_0,v_1)$ be any bond in $\mathcal{B}$, and let $\gamma_0$ be the periodic orbit in $\bar{\gamma}$ containing $(v_0,v_1)$. 
There will be at least one outgoing bond from $v_1$, since $v_1 \in o(\mathcal{B}) = t(\mathcal{B})$. 
If $v_1$ is not a repeated vertex, then there is only one bond which can follow $(v_0,v_1)$ in $\gamma_0$; if $v_1$ is repeated, then there are two outgoing bonds that may be chosen. 
In either case, the terminal vertex of this bond will then have either one or two outgoing bonds that can be chosen next, and the same will be true of each vertex until a vertex appears for a second time in $\gamma_0$; call this vertex $v$. 
If $v \neq v_0$, then $v$ must be repeated in $t(\mathcal{B})$ and thus in $o(\mathcal{B})$, so there is a second outgoing bond from $v$ that must be used at the repetition of $v$. 
If $v=v_0$, then $v_0$ may or may not be repeated in $o(\mathcal{B})$. 
If $v_0$ is not repeated, then the circuit corresponding to $\gamma_0$ must be completed; otherwise, there is a choice of whether to close the circuit, or to use the second outgoing bond and close the circuit the second time that a bond has terminal vertex $v_0$. 
In any case, this process will eventually complete a circuit, thus determining the corresponding periodic orbit $\gamma_0$. 
If $\gamma_0$ uses every bond of $\mathcal{B}$, then it forms a single-element pseudo orbit; otherwise, if $b_1$ is any bond in $\mathcal{B}$ that did not appear in $\gamma_0$, another periodic orbit $\gamma_1$ may be constructed by the same process. 
(The only difference is that a vertex that is repeated in $o(\mathcal{B})$, but was already used once in $\gamma_0$, will not offer a choice of outgoing bonds after it occurs in $\gamma_1$.) 
This process may be repeated until all the bonds in $\mathcal{B}$ are used, and this produces a pseudo orbit $\bar{\gamma}=\{\gamma_0,\gamma_1,\dots \}$. 
This process will always produce a pseudo orbit; there are exactly two subsequent choices one can make the first time that a bond is selected that has each of the $N$ repeated vertices as its terminal vertex (this may mean two choices of outgoing bonds, or if the vertex was the starting vertex of the construction, then the choice is to complete the circuit or use the remaining outgoing bond), and the choice of outgoing bond at any other vertex is already determined; so the number of possible pseudo orbits is $2^N$.
\end{proof}
\vspace{1em}

If $v$ is a vertex, assume a bond $(v\pm c_1,v\pm c_2)$ is interpreted as $(v\pm c_1 \mod n, v\pm c_2 \mod n)$ and a pair of vertices $v\pm c_1,v\pm c_2$ can be likewise interpreted as $v\pm c_1 \mod n, v\pm c_2 \mod n$, unless otherwise specified.

\section{Counting Primitive Periodic Orbits}
\label{sec:PO}

Primitive periodic orbits on a graph can be counted using the trace of the powers of the graph's adjacency matrix $A$ \cite{HH20}. 
However, the equation used in this method,
\begin{equation}
    \textrm{Tr}(A^l)=\sum_{\omega|l}\omega \cdot PO(\omega) = \underbrace{\sum_{\omega|l,\ \omega\neq l}\omega \cdot PO(\omega)}_\text{non-primitive circuits} + \underbrace{l \cdot PO(l)}_\text{primitive circuits} \ ,
\end{equation}
is recursive: in order to find the number of primitive periodic orbits of a given length $l$, one must first count the number of primitive periodic orbits for any length that is a divisor of $l$. 
We find a closed formula, that is not dependent on a graph's adjacency matrix, that counts the number of primitive periodic orbits on families of directed circulant graphs $C_n^+(a_1,a_1+d)$. 

\subsection{First Family of Graphs \texorpdfstring{$C_n^+(1,2)$}{}}

We first look at the family of graphs $C_n^+(1,2)$ whose adjacency matrices have first row $\begin{bmatrix}
0 & 1 & 1 & 0 & 0 & \dots & 0
\end{bmatrix}$.

\begin{proposition}
\label{thm:numPOfirstfamily}
For a graph $C_n^+(1,2)$, where $n>2$, the number of primitive periodic orbits of length $0<l\leq n$ is given by
\begin{equation}
    \text{PO}(n,l) = \begin{cases}
    \frac{n}{l} \binom{l}{n-l} & \text{if } l<n\\
    2 & \text{if } l=n, \text{ and } n \text{ is odd} \\
    1 & \text{if } l=n, \text{ and } n \text{ is even}
    \end{cases} \ . 
\end{equation}
\end{proposition}

\begin{proof}
Let $\mathcal{G}=C_n^+(1,2)$, where $n>2$. 
Any vertex of $\mathcal{G}$ has two outgoing bonds, one of arc 1 and one of arc 2.

Let $x$ be a circuit on $\mathcal{G}$ with length $0 < l \leq n$. 
Since the arc of each bond of $\mathcal{G}$ is at least 1 and at most 2, the walk sum $\alpha$ of $x$ is $0<\alpha\leq 2n$. 
Thus, by Lemma \ref{thm:circuitwalksum}, there are two cases: $\alpha=n$ or $\alpha=2n$.

Consider the case where $\alpha=n$. 
Let $B$ be the number of bonds of arc 2 that appear in $x$. 
Since all the bonds of $\mathcal{G}$ have arc 1 or 2, $\alpha=2B+1(l-B)=B+l$, so $B+l=n$ and $B=n-l$. 
Thus, exactly $n-l$ bonds of arc 2 appear in $x$. Since any vertex of $\mathcal{G}$ has an outgoing bond of arc 2, these can be any of the $l$ bonds in the circuit. 
Thus, the number of different circuits with origin vertex $v$ and walk sum $\alpha=n$ on $\mathcal{G}$ is $\binom{l}{n-l}$. 
Since there are $n$ possible origin vertices, the total number of circuits with walk sum $\alpha=n$ on $\mathcal{G}$ is $n \binom{l}{n-l}$. 

\begin{lemma}
\label{thm:numprimPOfirstfamily}
On a directed circulant graph, there exists no nonprimitive circuit and thus no nonprimitive periodic orbit with walk sum $n$.
\end{lemma}
\begin{proof}
Suppose, for the sake of contradiction, that a circuit $y$ on $\mathcal{G}$ is nonprimitive. 
Then $y$ is equal to some circuit $y'$ repeated at least twice. 
But since the walk sum of $y$ is $n$, the walk sum of $y'$ could then be no greater than $n/2$, and so cannot be a multiple of $n$, which contradicts Lemma \ref{thm:circuitwalksum}. 
Thus, there exists no nonprimitive circuit with $\alpha=n$ on $\mathcal{G}$.
\end{proof}
\vspace{1em}

Since each primitive periodic orbit of length $l$ consists of $l$ circuits, and there is no circuit that is nonprimitive, the total number of primitive periodic orbits with $\alpha=n$ is 
\begin{equation}
\label{eq:binom1stfam}
\frac{n}{l} \binom{l} {n-l} \ .
\end{equation}
We will see in the next case that walk sum $\alpha=2n$ is impossible for circuits of length $0<l<n$, so \eqref{eq:binom1stfam} is also the total number of primitive periodic orbits of length $0<l<n$.

Consider the case where $\alpha=2n$. 
This is only possible if $l=n$, and every bond that appears in $x$ has arc 2. 
If $n$ is even, then this produces the walk $v,v+2,\dots,n-2,0,\dots,v,\dots,n-2,0,\dots,v$ if $v$ is even, and $v,v+2,\dots,n-1,1,\dots,v,\dots,n-1,1,\dots,v$ if $v$ is odd. 
These are each one shorter circuit repeated twice, and thus nonprimitive. 
If $n$ is odd, then this produces the walk $v,v+2,\dots,n-1,1,\dots,v-1,v+1,\dots,n-2,0,\dots,v$ if $v$ is even, and $v, v+2, \dots, n-2, 0, \dots, v-1,v+1,\dots,n-1,1,\dots,v$ if $v$ is odd. 
In either case, $v$ does not appear in the middle, so it is a primitive circuit. 
Since there are $n$ possible origin vertices, there are $n$ primitive circuits, and since they are of length $n$, they form $n/n=1$ periodic orbit. 
Thus, the second case produces one primitive periodic orbit if $n$ is odd, and none if $n$ is even.

So, when $l=n$, the number of primitive periodic orbits is \eqref{eq:binom1stfam} evaluated when $l=n$ (which simplifies to one) plus the number of primitive periodic orbits produced by the $\alpha=2n$ case.
Thus, the result follows.
\end{proof}

\subsection{Second Family of Graphs \texorpdfstring{$C_n^+(1,3)$}{}}

We now consider a second family of graphs $C_n^+(1,3)$, whose adjacency matrices have first row $\begin{bmatrix}
0 & 1 & 0 & 1 & 0 & 0 & \dots & 0
\end{bmatrix}$.

\begin{proposition}
\label{thm:numPOsecondfamily}
For a graph $C_n^+(1,3)$, where $n>3$, the number of primitive periodic orbits of length $0<l\leq n$ is given by
\begin{equation}
    \text{PO}(n,l) = \begin{cases}
    \frac{n}{l}\left[\binom{l}{(n-l)/2}+ \binom{l} {(2n-l)/2}-\binom{l/2}{(2n-l)/4}\right] & \text{if } l<n\\
    \binom{l}{l/2}-\binom{l/2}{l/4} + 1 & \text{if } l=n \text{ and } 3|n\\
    \binom{l}{l/2}-\binom{l/2}{l/4} + 2 & \text{if } l=n \text{ and } 3 \not| \text{ n}\\
    \end{cases} \ .
\end{equation}
\end{proposition}
\begin{proof}
Let $\mathcal{G}=C_n^+(1,3)$, where $n>3$. Any vertex of $\mathcal{G}$ has two outgoing bonds, one of arc 1 and one of arc 3.

Let $x$ be a circuit of length $0< l \leq n$ on $\mathcal{G}$. 
By Lemma \ref{thm:circuitwalksum}, the walk sum $\alpha$ of $x$ is a multiple of $n$.
Since the arc of each bond of $\mathcal{G}$ is at least 1 and at most 3, the walk sum $\alpha$ of $x$ is $0<\alpha\leq 3n$. 
Thus, there are three cases: $\alpha=n$, $\alpha=2n$, or $\alpha=3n$.

Let $B$ be the number of bonds in $x$ of arc 3. Then $\alpha=3B+1(l-B)=2B+l$. 

Consider the case where $\alpha=n$. 
Then, $2B+l=n$ and $B=(n-l)/2$. 
Thus, exactly $(n-l)/2$ bonds of arc 3 appear in $x$. 
(If $(n-l)/2$ is not an integer, then it is not possible to have a circuit with walk sum $n$, since this would represent a noninteger number of bonds.) 
Since any vertex of $\mathcal{G}$ has an outgoing bond of arc 3, these can be any of the $l$ bonds in the circuit. 
Thus, this case produces 
\begin{equation}
\label{eq:walksumn}
    \binom{l} {\frac{n-l}{2}}
\end{equation} 
different circuits with origin vertex $v$. 
When $(n-l)/2$ is a noninteger, we interpret this combination to be zero. 
We use the same interpretation for the other combinations. 
By Lemma \ref{thm:numprimPOfirstfamily}, there exists no nonprimitive circuit in this case.

Consider the case where $\alpha=2n$. 
Then $2B+l=2n$ and $B=(2n-l)/2$. 
Thus, exactly $(2n-l)/2$ bonds of arc 3 appear in $x$. 
Since these can be any of the $l$ bonds in the circuit, this case produces 
\begin{equation}
\binom{l}{\frac{2n-l}{2}}
\end{equation}
different circuits with origin vertex $v$.

Let $y$ be a nonprimitive circuit of length $l$ and walk sum $2n$ on $\mathcal{G}$. 
Then $y$ must consist of a shorter circuit $y'$ repeated at least twice. 
Thus this shorter circuit $y'$ must have a walk sum $\alpha_{y'}=n$ or $\alpha_{y'}=2n$. 
But $\alpha_y=r\alpha_{y'}$, where $r$ is the number of times $y'$ is repeated, which is at least twice, and $\alpha_y=2n$. 
The only way this is possible is if $\alpha_{y'}=n$, with $y'$ being repeated twice. 

The circuit $y'$ then has length $l/2$, which means the number of bonds of arc 3 must be 
\begin{equation}
    \frac{n-(l/2)}{2}=\frac{2n-l}{4} \ .
\end{equation} 
Since these can be any of the $l/2$ bonds in $y'$, the number of possible circuits of this length with origin vertex $v$ is 
\begin{equation}
\label{eq:walksum2nnonprim}
    \binom{\frac{l}{2}}{\frac{2n-l}{4}} \ .
\end{equation} 
Each of these produces exactly one nonprimitive circuit of length $l$, in which $y'$ is repeated twice; thus, the number of nonprimitive circuits on $\mathcal{G}$ of length $l$ with origin vertex $v$ is \eqref{eq:walksum2nnonprim}. 
In the case where $l/2$ is a noninteger, we interpret the combination to be  zero, since a circuit must have an integer length. 
We use the same interpretation for the other combinations.

Consider the case where $\alpha=3n$. 
This is only possible if $l=n$. 
Every bond that appears in $x$ must have arc 3, which means there is only one circuit from any origin vertex, for a total of $n$ circuits.

There are two cases: $3|l$, or $3\not| l$. 
If $3|l$, then for length $l/3=n/3$, \eqref{eq:walksumn} gives $\binom{n/3}{n/3} = 1$
primitive circuit with origin vertex $v$ and walk sum $n$. 
Since there are $n$ possible origin vertices and each primitive periodic orbit of length $n/3$ consists of $n/3$ circuits, the total number of primitive periodic orbits of length $n/3$ is $n/(n/3)$; that is, three primitive periodic orbits of length $n/3$ with walk sum $n$. 
Since they are primitive, they each contain $n/3$ circuits, for a total of $3(n/3)=n$ primitive circuits of length $n/3$ and walk sum $n$. 
Each of these $n$ circuits can be repeated three times to produce a nonprimitive circuit of length $n$ and walk sum $3n$. 
Since there are $n$ circuits total of length $n$ and walk sum $3n$, they must all be nonprimitive, and so no primitive orbits are added in this case.

Now consider the second case.
\begin{lemma}
\label{thm:numprimPOsecondfamily}
On a directed circulant graph $C_n^+(1,3)$ such that $3 \not| n$, there is one primitive periodic orbit of length $l=n$ with walk sum $3n$.
\end{lemma}
\begin{proof}
Since $l=n$ is not divisible by 3, there is one primitive circuit starting at 0 with a walk sum of $3n$, which takes the form $0,3,\dots,n-2,1,4,\dots,n-1,2,5,\dots,n-3,0$ if $n-2$ is divisible by 3, or $0,3,\dots,n-1,2,5,\dots,n-2,1,4,\dots,n-3,0$ if $n-1$ is divisible by 3. 
As the primitive circuit in either case contains $n$ unique vertices, there is a single primitive periodic orbit with walk sum $3n$.
\end{proof}
\vspace{1em}

Consider the case where $l<n$. 
The number of primitive circuits on $\mathcal{G}$ of length $l$ with origin vertex $v$ is the total number of circuits with walk sums $n$ and $2n$ minus the number of nonprimitive circuits with walk sums $n$ and $2n$. 
Since there are $n$ possible origin vertices, and each primitive periodic orbit of length $l$ consists of $l$ circuits, the total number of primitive periodic orbits on $\mathcal{G}$ is
\begin{equation}
\label{eq:totalPOfirstfamily}
    \frac{n}{l}\left[\binom{l}{\frac{n-l}{2}}+\binom{l}{\frac{2n-l}{2}}- \binom{\frac{l}{2}}{\frac{2n-l}{4}}\right] \ .
\end{equation}

Consider the case where $l=n$. 
With the addition of the possible walk sum of $3n$, the number of primitive periodic orbits is
\begin{align}
\frac{n}{l}\left[\binom{l}{\frac{n-l}{2}}+\binom{l}{\frac{2n-l}{2}}-\binom{\frac{l}{2}}{\frac{2n-l}{4}}\right] & = 1\left[\binom{l}{0}+\binom{l}{\frac{l}{2}}-\binom{\frac{l}{2}}{\frac{l}{4}}\right]\\
    & = \binom{l}{\frac{l}{2}}-\binom{\frac{l}{2}}{\frac{l}{4}} + 1 \ ,
\end{align} 
if $3|n$, and 
\begin{align}
 \frac{n}{l}\left[\binom{l}{\frac{n-l}{2}}+\binom{l} {\frac{2n-l}{2}}-\binom{\frac{l}{2}}{\frac{2n-l}{4}}\right]+1 & = 1\left[\binom{l}{0}+\binom{l}{\frac{l}{2}}-\binom{\frac{l}{2}}{\frac{l}{4}}\right]+1\\
    & = \binom{l}{\frac{l}{2}}-\binom{\frac{l}{2}}{\frac{l}{4}} + 2 \ ,
\end{align}
if $3\not|n$. 
Thus the result follows.
\end{proof}

\subsection{General Families of Graphs \texorpdfstring{$C_n^+(a_1,a_1+d)$}{}}
\label{sec:generalresult}

We now state the main counting result for primitive periodic orbits on a 4-regular directed circulant graph $C_n^+(a_1,a_1+d)$, where $a_1, d \in \mathbb{N}$.
The result has been verified numerically for several families of graphs and up to periodic orbit length 15; see appendix A.

In the following theorem, we will use the M\"obius function
\begin{equation}
    \mu(\omega)=\begin{cases} 1 & \text{if } \omega=1\\
    0 & \text{if } \omega \text{ is divisible by the square of a prime number}\\
    (-1)^k & \text{otherwise, } k \text{ is the number of prime factors of } \omega
    \end{cases} \ .
\end{equation}
See \cite{LP98} for more details.

\begin{theorem}
\label{thm:generalPO}
For a graph $C_n^+(a_1,a_1+d)$, where $n>a_1+d$, the number of primitive periodic orbits of length $l>0$ is given by  
\begin{equation}
    PO(n,l,a_1,d)= \frac{n}{l} \sum_{m=\lceil\frac{a_1l}{n}\rceil}^{\lfloor\frac{(d+a_1)l}{n}\rfloor} \sum_{\omega|m \text{, } \omega|l} \mu(\omega) \binom{\frac{l}{\omega}}{\frac{mn-a_1l}{\omega d}} \ .
\end{equation}
\end{theorem}

\begin{proof}
Let $\mathcal{G}=C_n^+(a_1,a_1+d)$, where $n>a_1+d$. 
Each vertex of $\mathcal{G}$ has exactly two outgoing bonds, one of arc $a_1$ and one of arc $a_1+d$.

Let $x$ be a circuit on $\mathcal{G}$ with length $l$ and walk sum $\alpha$. 
By Lemma \ref{thm:circuitwalksum}, $\alpha=mn$ for some $m\in\mathbb{N}$. 
Let $B$ be the number of bonds of arc $a_1+d$ that appear in $x$. 
Thus, $mn=B(a_1+d)+a_1(l-B)$ and $B=(mn-a_1l)/d$. 
Since any vertex has an outgoing bond of arc $a_1+d$, these can be any of the $l$ bonds in $x$. 
Thus, the number of different circuits on $\mathcal{G}$ with origin vertex $v$, length $l$, and walk sum $mn$ is
\begin{equation}
    \binom{l}{\frac{mn-a_1l}{d}} \ .
\end{equation}

Let $\omega$ be a common divisor of $m$ and $l$. 
The number of circuits with origin vertex $v$, length $l/\omega$, and walk sum $mn/\omega$ is 
\begin{equation}
\label{eq:combination}
    \binom{\frac{l}{\omega}}{\frac{mn-a_1l}{\omega d}} \ .
\end{equation}

Consider the sum
\begin{equation}
\label{eq:gensum}
    \sum_{\omega|m \text{, } \omega|l} \mu(\omega) \binom{\frac{l}{\omega}}{\frac{mn-a_1l}{\omega d}} \ .
\end{equation}

Since $\mu(1)=1$, when $\omega=1$, 
\begin{equation}
    \mu(\omega) \binom{\frac{l}{\omega}}{\frac{mn-a_1l}{\omega d}}= \binom{l}{\frac{mn-a_1l}{d}} \ ,
\end{equation}
which is equal to the total number of circuits with origin vertex $v$, length $l$, and walk sum $mn$.

When $\omega>1$, the combination \eqref{eq:combination} represents the number of circuits of length $l$ that consist of a shorter circuit repeated $\omega$ times.
 
If $x$ is a primitive circuit, it will only add one term to the sum, where $\omega=1$. 
Otherwise, let $q$ be the largest possible number such that the nonprimitive circuit $x$ consists of a circuit of length $l/q$ repeated $q$ times. 
The circuit $x$ is also formed of a circuit of length $l/\omega$ repeated $\omega$ times, for each divisor $\omega$ of $q$. 
Now $q$ must be a divisor of $m$, since the shorter circuit has walk sum $mn/q$, which by Lemma \ref{thm:circuitwalksum}, must be a multiple of $n$. 
Since $q$ is a divisor of $m$, all of its unique prime factors are also factors of $m$; call these $p_1,\dots ,p_j$. 
For each divisor $\omega$ of $q$, there are two cases: $\omega$ is divisible by the square of one of the prime factors, or it is the product of a subset of these factors. 
In the first case, $\mu(\omega)=0$ and so nothing is added to the sum.

Consider the second case. 
The product of any subset of $\{p_1,\dots ,p_j\}$ is a divisor of $q$ and is not divisible by the square of a prime number, so $\mu(\omega)\neq 0$. 
Therefore, the circuit will add exactly one term to the sum for every subset $X \subset \{p_1,\dots,p_j\}$, with $\omega=\prod_{g\in X}g$.

For each subset $Y$ such that $p_1 \in Y$, there will be a corresponding subset $Z$ such that $Z= Y - \{p_1\}$. 
Let $P_Y=\prod_{g \in Y}g$ and $P_Z=\prod_{g \in Z}g$. 
$P_Y$ and $P_Z$ are the possible values of $\omega$ corresponding to the subsets $Y$ and $Z$. 
If $Z$ has $h$ elements, $Y$ has $h+1$ elements. 
Thus, if $h>0$, $\mu(P_Y)=(-1)\mu(P_Z)$; this also holds if $h=0$ since $\mu(P_Y)=\mu(p_1)=-1$ and $\mu(P_Z)=\mu(1)=1$. 
(Here, we use the convention that the product over the empty set is equal to 1.) 
In any case, the circuit is added to the sum once and subtracted from the sum once, so these two terms cancel. 
All of the subsets of $\{p_1,\dots ,p_j\}$ form pairs $Y,Z$, so $x$ is not counted. 
Thus, \eqref{eq:gensum} counts the number of primitive circuits with origin vertex $v$, length $l$, and walk sum $mn$.

By Lemma \ref{thm:circuitwalksum}, every primitive circuit has a walk sum $mn$ for some $m \in \mathbb{N}$, where $(d+a_1)l/n\geq m \geq a_1l/n$, since the greatest possible walk sum is $(d+a_1)l$ and the least possible walk sum is $a_1l$. 
Thus, the total number of primitive circuits on $\mathcal{G}$ with origin vertex $v$ and length $l$ is
\begin{equation}
    \sum_{m=\lceil\frac{a_1l}{n}\rceil}^{\lfloor\frac{(d+a_1)l}{n}\rfloor}\sum_{\omega|m \text{, } \omega|l} \mu(\omega) \binom{\frac{l}{\omega}}{\frac{mn-a_1l}{\omega d}} \ .
\end{equation}

Since there are $n$ possible origin vertices, and there are $l$ circuits in each periodic orbit of length $l$, the result holds.
\end{proof}

\section{Counting Primitive Pseudo Orbits}
\label{sec:PsO}

We now return to our first two families, $C_n^+(1,2)$ and $C_n^+(1,3)$, to investigate their primitive pseudo orbits. 
Through a series of propositions, we count the number of primitive pseudo orbits with no self-intersection, as well as those with any given number of 2-encounters of length zero, for these two families.
The results are verified numerically up to periodic orbit length $n$ for several graphs in the second family; see appendix B.

\subsection{First Family of Graphs \texorpdfstring{$C_n^+(1,2)$}{}}

For the first family, we count the total number of primitive pseudo orbits for $0<l\leq n$ and show that they all have no self-intersection.

\begin{proposition}
\label{thm:numPsOfirstfamily}
For a graph $C_n^+(1,2)$, where $n>2$, the number of primitive pseudo orbits of length $0<l\leq n$ is given by
\begin{equation}
    \text{PsO}(n,l) = \begin{cases}
    \frac{n}{l} \binom{l} {n-l} & \text{if } l<n\\
    2 & \text{if } l=n\\
    \end{cases} \ .
\end{equation}
\end{proposition}
\begin{proof}
Let $\mathcal{G}=C_n^+(1,2)$, with $n>2$.

Suppose $0<l<n/2$. Then $n-l>l$. Thus, by Proposition \ref{thm:numPOfirstfamily}, $PO(n,l)=0$.

\begin{lemma}
\label{thm:primPsOwalksumfirstfamily}
For a graph $C_n^+(1,2)$, where $n>2$, any primitive pseudo orbit of length $0 < l < n$ consists of exactly one primitive periodic orbit of length $l$ and walk sum $n$.
\end{lemma}
\begin{proof}
Let $\bar{\gamma}$ be a primitive pseudo orbit of length $0<l<n$ and walk sum $\alpha$ on $\mathcal{G}$. 
For a pseudo orbit to consist of more than one primitive periodic orbit, one of the periodic orbits would have to have a length of $l/2$ or less, and so less than $n/2$. 
Since there are no such primitive periodic orbits, every primitive pseudo orbit of length $0<l<n$ must consist of exactly one primitive periodic orbit of length $l$ and walk sum $\alpha$. 
Then $0< \alpha \leq 2l < 2n$, so by Lemma \ref{thm:circuitwalksum}, $\alpha =n$.
\end{proof}
\vspace{1em}

Each primitive periodic orbit forms a primitive pseudo orbit containing a single periodic orbit; thus, by Lemma \ref{thm:primPsOwalksumfirstfamily}, the number of primitive pseudo orbits of length $l$, for $0<l<n$, is equal to the number of primitive periodic orbits, $PO(n,l)$.

If $l=n$, and $n$ is odd, since $n/2$ is not an integer there are still no primitive periodic orbits with length $l/2$ or less, and so the number of primitive pseudo orbits is the same as the number of primitive periodic orbits, which is 2. 
If $n$ is even, the number of primitive periodic orbits of length $n/2$ is 
\begin{equation}
PO\left( n,\frac{n}{2} \right) = \frac{n}{\frac{n}{2}} \binom{\frac{n}{2}}{n-\frac{n}{2}} = 2 \ .
\end{equation}
So there is one way of choosing two of these periodic orbits. 
Then there is exactly one pseudo orbit consisting of two primitive periodic orbits of length $n/2$. 
Since there is one primitive periodic orbit of length $n$, there is also one pseudo orbit consisting of this periodic orbit. 
In either case, when $l=n$, the number of primitive pseudo orbits is 2.
\end{proof}

\begin{proposition}
\label{thm:noselfintfirstfamily}
For a graph $C_n^+(1,2)$, with $n>2$, no primitive pseudo orbit of length $0<l\leq n$ has any self-intersection.
\end{proposition}
\begin{proof}
Let $\mathcal{G}=C_n^+(1,2)$, with $n>2$.

Let $\bar{\gamma}$ be a primitive pseudo orbit on $\mathcal{G}$ with length $0<l\leq n$ and walk sum $\alpha$. 
Since $0< \alpha \leq 2n$, by Lemma \ref{thm:circuitwalksum}, either $\alpha=n$ or $\alpha=2n$.

Consider the case where $\alpha=n$. 
By Lemma \ref{thm:circuitpassing}, $\bar{\gamma}$ can only pass each vertex of $\mathcal{G}$ once, which means there can be no repeated vertices and thus no self-intersection; thus, any primitive pseudo orbit of length $0<l\leq n$ and walk sum $n$ has no self-intersection.

In the case where $\alpha=2n$, which only occurs when $l=n$, there are two subcases: $n$ is odd or $n$ is even. 
If $n$ is odd, there is only one pseudo orbit with $\alpha=2n$. 
This pseudo orbit consists of the periodic orbit corresponding to $0,2,\dots,n-1,1,3,\dots,n-2,0$, which contains each vertex exactly once and so has no self-intersection. 
If $n$ is even, there is again one pseudo orbit with $\alpha=2n$, consisting of the periodic orbits corresponding to $0,2,\dots,n-2,0$ and $1,3,\dots,n-1,1$. 
These do not share any vertices or contain any repeated vertices, and so this pseudo orbit also has no self-intersection.

In every case, there exists no pseudo orbit on $\mathcal{G}$ of length $0<l\leq n$ with any self-intersection.
\end{proof}

\subsection{Second Family of Graphs \texorpdfstring{$C_n^+(1,3)$}{}}

For the second family, since some of the pseudo orbits of length $0<l\leq n$ do have self-intersections, we directly count the number of pseudo orbits with no self-intersection and with a given number of 2-encounters of length zero.

\begin{proposition}
\label{thm:numPsOnoselfintsecondfamily}
For a graph $C_n^+(1,3)$, with $n>3$, the number of pseudo orbits of length $0<l\leq n$ with no self-intersection is given by 
\begin{equation}
    \text{PsO$_0$}(n,l) = \begin{cases}
    \frac{n}{l} \binom{l}{(n-l)/2} + \frac{2n}{l} \binom{l/2}{n-l} & \text{if } l<n\\
    2 & \text{if } l=n, \text{ and } $n$ \text{ is odd} \\
    4 & \text{if } l=n, \text{ and } $n$ \text{ is even}
    \end{cases} \ .
\end{equation}
\end{proposition}
\begin{proof}
Let $\mathcal{G}=C_n^+(1,3)$, with $n>3$.

Let $\bar{\gamma}$ be a primitive pseudo orbit on $\mathcal{G}$ of length $l$ and walk sum $\alpha$. 
Since $0< \alpha \leq 3n$, by Lemma \ref{thm:circuitwalksum}, there are three possible cases: $\alpha=n,\alpha=2n$, and $\alpha=3n$.

In the first case where $\alpha=n$, the only way to construct a primitive pseudo orbit of length $l$ is with one primitive periodic orbit of length $l$ and walk sum $n$. 
Thus, the number of primitive pseudo orbits of length $l$ produced by this case is equal to the number of primitive periodic orbits of length $l$ and walk sum $n$, which was shown in Proposition \ref{thm:numPOsecondfamily} to be
\begin{equation}
\label{eq:numprimwalkn}
    \frac{n}{l} \binom{l}{\frac{n-l}{2}} \ .
\end{equation}
Since $\alpha=n$, $\bar{\gamma}$ passes any vertex $v$ only once by Lemma \ref{thm:circuitpassing}, which means $v$ cannot appear more than once; therefore, it has no self-intersection. 
Thus, the number of primitive pseudo orbits produced by this case with no self-intersection is \eqref{eq:numprimwalkn}.

Consider the second case where $\alpha=2n$. 
If $B$ is the number of bonds of arc 1 that appear in $\bar{\gamma}$, then $\alpha=B+3(l-B)=3l-2B$. 
Thus, $2n=3l-2B$, and $B=(3l-2n)/2$. 
(Note that this is not an integer if $l$ is odd. 
Thus, a pseudo orbit with $\alpha=2n$ must have an even length.)

We now construct all possible pseudo orbits with $\alpha=2n$.
%Original: We now construct all possible pseudo orbits with $\alpha=2n$, while showing that all possible pseudo orbits can be constructed in this way.  Better rephrase?

Consider the pairs of vertices of the form $v,v+1$ where the bond $(v,v+1)$ does not appear in $\bar{\gamma}$. 
The number of such pairs is $n-B=n-(3l-2n)/2=3l/2$. 
Since $l\geq 1$, there is at least one pair. There are three possible categories of pairs: one where neither $v$ nor $v+1$ appears in $\bar{\gamma}$, one where exactly one of $v$ and $v+1$ appears in $\bar{\gamma}$, and one where both $v$ and $v+1$ appear in $\bar{\gamma}$. 
Since $\alpha =2n$, by Lemma \ref{thm:circuitpassing}, $\bar{\gamma}$ must pass $v-1$ exactly twice. 
In order for $\bar{\gamma}$ to pass $v-1$ twice, the bond $(v-3,v)$, the bond $(v-2,v+1)$, or the vertex $v-1$ must appear in $\bar{\gamma}$ twice (collectively). 
If $v$ is not included in $\bar{\gamma}$, the bond $(v-3,v)$ cannot appear and if $v+1$ does not appear in $\bar{\gamma}$, the bond $(v-2,v+1)$ cannot appear. 
Thus $v-1$ must appear in $\bar{\gamma}$ twice, meaning the orbit would have a self-intersection. 
Therefore, there are no pairs of vertices in the first category for a pseudo orbit $\bar{\gamma}$ with no self-intersection.

Consider the second category. 
Let $u$ be one of the $n-l$ vertices that do not appear in $\bar{\gamma}$. 
There will be two pairs of vertices of the form $v,v+1$ containing $u$, and for each of these the bond $(v,v+1)$ does not appear in $\bar{\gamma}$. 
These pairs must contain only the one vertex that does not appear in $\bar{\gamma}$, since the case where both do not appear is impossible; this means that any vertex that is adjacent to $u$ via a bond of the form $(v,v+1)$ appears in $\bar{\gamma}$. 
Then there are $2(n-l)$ pairs of vertices that fall under the second category.

There are $n$ total pairs of the form $v,v+1$, so, given the number of pairs in the first two categories, the number of pairs in the third category, where both vertices appear in $\bar{\gamma}$, but the bond between them is not, is
\begin{equation}
n-\frac{3l-2n}{2}-2(n-l) = n-\frac{2n-l}{2} = \frac{l}{2} \ .
\end{equation}
Thus, any pseudo orbit with walk sum $2n$ must contain exactly $l/2$ pairs of consecutive vertices $v,v+1$ such that the bond $(v,v+1)$ does not appear in $\bar{\gamma}$, which is at least one pair since $l \geq 2$.

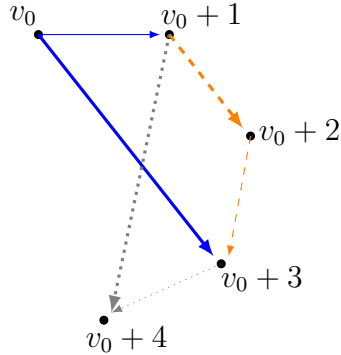
\begin{figure}[!htb]
\centering
\begin{tikzpicture}

% vertices
\draw[fill=black] (-0.87,1.8) circle (1.5pt);
\draw[fill=black] (0.87,1.8) circle (1.5pt);
\draw[fill=black] (1.95,0.45) circle (1.5pt);
\draw[fill=black] (1.56,-1.25) circle (1.5pt);
\draw[fill=black] (0,-2) circle (1.5pt);
% vertex labels
\node at (-1.07,2.05) {$v_0$};
\node at (1.3,2.05) {$v_0 +1$};
\node at (2.6,0.5) {$v_0 +2$};
\node at (2.1,-1.45) {$v_0 +3$};
\node at (0.3,-2.25) {$v_0 +4$};
% bonds
\draw[blue, -latex] (-0.87,1.8) -- (0.75,1.8);
\draw[orange, dashed, very thick, -latex] (0.87,1.8) -- (1.85,0.55);
\draw[orange, dashed, -latex] (1.95,0.45) -- (1.66,-1.15);
\draw[gray, dotted, -latex] (1.56,-1.25) -- (0.1,-1.9);
\draw[blue, very thick, -latex] (-0.87,1.8) -- (1.46,-1.15);
\draw[gray, dotted, very thick, -latex] (0.87,1.8) -- (0.1,-1.9);

\end{tikzpicture}
\caption{The vertex pair $v_0$, $v_0 +1$ is of the third category; both vertices appear in $\bar{\gamma}$, but the bond $(v_0, v_0 +1)$ does not (solid/blue), so the only outgoing bond at $v_0$ that appears in $\bar{\gamma}$ is $(v_0, v_0 +3)$ (bold solid/blue). 
There are two options for outgoing bonds at $v_0 + 1$ (bold dashed/orange and bold dotted/gray), but for either one that is selected, there is a corresponding choice of bond that cannot simultaneously appear in $\bar{\gamma}$ (dashed/orange and dotted/gray), in order to avoid self-intersections in $\bar{\gamma}$.}
\label{fig:bondselection}
\end{figure}

Take any such pair of vertices, $v_0,v_0+1$, and consider the two outgoing bonds at $v_0$ and $v_0+1$. 
The outgoing bond at $v_0$ must be $(v_0,v_0+3)$, since $(v_0,v_0+1)$ is not used. 
The outgoing bond at $v_0+1$ may be either $(v_0+1,v_0+2)$ or $(v_0+1,v_0+4)$. 
In either case, the pair of terminal vertices for these bonds, either $v_0+2,v_0+3$ or $v_0+3,v_0+4$, takes the form $v_1,v_1+1$, and like $v_0,v_0+1$, falls into the third category: both vertices appear in $\bar{\gamma}$, and, since they both already have incoming bonds, the bond between them cannot appear in $\bar{\gamma}$; see figure \ref{fig:bondselection}.
Since either $(v_0+1,v_0+2)$ is used as a bond, or $v_0+2$ is passed by both bonds (in which case it does not appear in $\bar{\gamma}$ since it has already been passed twice), there cannot be another such pair in this form between $v_0,v_0+1$ and the pair of terminal vertices. 
Thus, all $l/2$ of the pairs in the third category may be found in progression by this process, with a pair of bonds connecting each pair of vertices to the next pair of vertices, and the final pair of vertices connected by a pair of bonds to the first pair of vertices. 
Because of the way they are constructed, each of these pairs of bonds will either take the form $(v,v+3)$ and $(v+1,v+2)$, or $(v,v+3)$ and $(v+1,v+4)$. 
The only way for a vertex not to appear in $\bar{\gamma}$ is when the second form, which passes the single vertex $v+2$ twice, does appear. 
Then the vertex $v+2$ cannot appear elsewhere in the orbit, by Lemma \ref{thm:circuitpassing}. 
Since there must be $n-l$ vertices that are not used, this second form must be used $n-l$ times, and the first form is used for the remaining pairs. 
Let $\mathcal{B}$ be the set of bonds selected in these pairs. 
The bonds of $\mathcal{B}$ create a walk from each of the starting vertices $v_0,v_0+1$; the walk sum is 4 for each pair plus 2 extra for each pair of the second form, which sums to $4(l/2)+2(n-l)=2l+2n-2l=2n$; thus, the final pair is then $v_0+n,v_0+n+1$, which is equivalent to $v_0,v_0+1$. 
Then the set of origin vertices of $\mathcal{B}$ and the set of terminal vertices of $\mathcal{B}$ are the same, so by Lemma \ref{thm:bondstoPsO}, at least one pseudo orbit can be formed from the bonds of $\mathcal{B}$; this pseudo orbit has length $2(l/2)=l$ and walk sum $2n$, and thus is a possible choice for $\bar{\gamma}$. 
Since there are $l/2$ pairs of bonds in a circuit of length $l$, the $n-l$ pairs of the second form can be chosen in $\binom{l/2}{n-l}$ ways from a given starting pair $v_0,v_0+1$. 
(We can interpret this combination to be equal to 0 if $l/2$ is not an integer, since the case where the walk sum of $\bar{\gamma}$ is $2n$ is impossible for odd $l$.) 
This decides all $l$ of the bonds in $\mathcal{B}$. 
Since there is no self-intersection, there are no repeated vertices, so by Lemma \ref{thm:bondstoPsO}, there is one pseudo orbit for each choice of bonds; thus, there are $\binom{l/2}{n-l}$ pseudo orbits that can be constructed from the starting pair $v_0,v_0+1$.

There are $n$ possible starting pairs of vertices on $\mathcal{G}$, one for each vertex that can be chosen as $v_0$, and there are $l/2$ different pairs in each pseudo orbit $\bar{\gamma}$ which could have been used as the starting pair $v_0,v_0+1$ (that is, both vertices appear in $\bar{\gamma}$ but the bond between them does not), so the total number of primitive pseudo orbits with no self-intersections produced by this case is
\begin{equation}
\label{eq:numprimwalk2n}
\frac{2n}{l} \binom{\frac{l}{2}}{n-l} \ .
\end{equation}

Consider the case where $\alpha=3n$, which can only occur if $l=n$. 
In this case, every bond must have arc 3, and so all $n$ bonds of arc 3 appear in $\bar{\gamma}$. 
If $n$ is not divisible by 3, by Lemma \ref{thm:numprimPOsecondfamily}, there is one primitive periodic orbit with walk sum $3n$, which is the only way to use all of the bonds of arc 3; thus, the pseudo orbit consisting of this one periodic orbit is the only one added in this case. 
If $n$ is divisible by 3, then these bonds create three periodic orbits of length $n/3$, corresponding to the circuits,
\begin{equation}
    0,3,\dots,n-3,0, \qquad 1,4,\dots,n-2,1, \quad \textrm{and} \quad 2,5,\dots,n-1,2 \ .
\end{equation}
Since the three periodic orbits have no intersection, there is exactly one pseudo orbit consisting of the three of them, which has no self-intersection, and their combined length is $l=n$; so again, there is one pseudo orbit added. 
Either way, the case where $\alpha =3n$ produces exactly one primitive pseudo orbit with no self-intersection.

The total number of primitive pseudo orbits with no self-intersection, if $0<l<n$, is given by the sum of the first two cases, \eqref{eq:numprimwalkn} and \eqref{eq:numprimwalk2n}.

If $l=n$, one must be added for the third case, and so the number is
\begin{equation}
    \frac{n}{l} \binom{n}{\frac{n-l}{2}} + \frac{2n}{l} \binom{\frac{l}{2}} {n-l} +1 \ .
\end{equation}
When $n$ is odd, 
\begin{equation}  
\frac{n}{l} \binom{n}{\frac{n-l}{2}} + \frac{2n}{l} \binom{\frac{l}{2}} {n-l} +1 =  1 \binom{n}{0} + 2 (0) +1 =  2 \ .  
\end{equation}
When $n$ is even,
\begin{equation}  
\frac{n}{l} \binom{n}{\frac{n-l}{2}} + \frac{2n}{l} \binom{\frac{l}{2}} {n-l} +1 =  1 \binom{n}{0} + 2 \binom{\frac{n}{2}} {0} +1 =  4  \ .
\end{equation}
\end{proof}

\begin{proposition}
\label{thm:numPsOsecondfamily2enc}
For a graph $C_n^+(1,3)$, with $n>3$, the number of pseudo orbits of length $0<l\leq n$ with $N>0$ 2-encounters of length zero, and no other self-intersection, is given by
\begin{equation}
\label{eq:numPsOsecondfamily2enc}
    PsO_N(n,l)=2^N \frac{n}{N} \binom{\frac{l}{2}-N}{n-l+N} \binom{\frac{l}{2}-N-1}{N-1} \ .
\end{equation}
\end{proposition}
\begin{proof}
Let $\mathcal{G}=C_n^+(1,3)$, with $n>3$.

Let $\bar{\gamma}$ be a pseudo orbit on $\mathcal{G}$ with length $0<l\leq n$ and walk sum $\alpha$. 
Since $0<\alpha \leq 3n$, by Lemma \ref{thm:circuitwalksum}, the only possible values of $\alpha$ are $n$, $2n$, and $3n$.

Consider the case where $\alpha=n$. 
By Lemma \ref{thm:circuitpassing}, $\bar{\gamma}$ passes each vertex exactly once, meaning it cannot use any vertex more than once, and so has no self-intersection. 
Thus, this case produces no pseudo orbits with 2-encounters of length zero.

Consider the case where $\alpha=3n$, which can only occur when $l=n$. 
Then all the bonds of $\bar{\gamma}$ must have arc 3, meaning that any vertex $v$ used in $\bar{\gamma}$ must be preceded by the incoming bond $(v-3,v)$ and followed by the outgoing bond $(v,v+3)$ in $\bar{\gamma}$. 
If there is any self-intersection in $\bar{\gamma}$, then some vertex $v$ is used more than once, which means the bonds $(v-3,v)$ and $(v,v+3)$ must also be used more than once, creating a 2-encounter of length at least two. 
Thus $\bar{\gamma}$ has no self-intersection, a 2-encounter of length greater than zero, or an $\ell$-encounter where $\ell>2$. 
So this case also produces no relevant pseudo orbits, and we only need to consider the case where $\alpha=2n$.

Suppose $\alpha=2n$. 
If $\bar{\gamma}$ has at least one 2-encounter of length zero, at least one vertex appears twice; let $v_0$ be such a vertex. 
Then both the bonds $(v_0,v_0+1)$ and $(v_0,v_0+3)$ appear in $\bar{\gamma}$. 
Consider the vertex $v_0+1$. This vertex cannot appear twice, since it is also passed by $(v_0,v_0+3)$, and so would be passed three times, which contradicts Lemma \ref{thm:circuitpassing} since $\alpha =2n$. 
Then $v_0+1$ is only used once, meaning exactly one of the two bonds $(v_0+1,v_0+2)$ and $(v_0+1,v_0+4)$ appears in $\bar{\gamma}$. 
This, together with $(v_0,v_0+3)$, creates a pair of bonds from a starting vertex pair $v_0,v_0+1$ to another pair of vertices $v-1,v$, which is equal to either $v_0+2,v_0+3$ or $v_0+3,v_0+4$. 
The vertex $v-1$ cannot appear a second time in $\bar{\gamma}$, since it is passed by the bond $(v-3,v)$; therefore, only one of the bonds $(v-1,v)$ and $(v-1,v+2)$ may appear in $\bar{\gamma}$. 
If $(v-1,v)$ appears, then $v$ appears twice, since both incoming bonds at $v$ appear in $\bar{\gamma}$. 
If $(v-1,v+2)$ is used instead, then $v$ cannot appear twice, since it has already been passed twice, by its first appearance and by $(v-1,v+2)$. 
Thus, exactly one of the bonds $(v,v+1)$ and $(v,v+3)$ appears. 
This creates a pair of bonds from $v-1,v$ to another pair of vertices, either $v+1,v+2$ or $v+2,v+3$, again of the form $v'-1,v'$. 
Again, if the bond between these vertices appears in $\bar{\gamma}$, then $v'$ appears twice; if not, another pair of bonds can be found, and the process continued until a vertex that appears twice is reached. 
Each of the pairs of bonds used in this construction either takes the form $(v-1,v+2),(v,v+1)$, or the form $(v-1,v+2),(v,v+3)$.

Call the first repeated vertex found by this process $v_1$. 
Then we have a set of bonds which form two walks from $v_0$ to $v_1$: $(v_0,v_0+1)$, one or more pairs of bonds connecting vertex pairs between $v_0,v_0+1$ and $v_1-1,v_1$, and $(v_1-1,v_1)$. 
By the construction, no vertex between $v_0$ and $v_1$ appears more than once, and so $v_1$ must be the next repeated vertex after $v_0$. 
In addition, this finds all the bonds that appear in $\bar{\gamma}$ between $v_0$ and $v_1$, since each vertex from $v_0+1$ to $v_1-1$ is passed twice. 
Let $N$ be the number of 2-encounters of length zero in $\bar{\gamma}$. 
If $N=1$, then $v_1=v_0$, and so every vertex of $\mathcal{G}$ is passed twice and there can be no other bonds in $\bar{\gamma}$. 
Otherwise, there is a set of bonds which forms two walks from $v_0$ to $v_1$, another set which forms two walks between $v_1$ and the next repeated vertex, and so on until a final set which forms two walks from a repeated vertex to $v_0$. 
Let $\mathcal{B}$ be the union of these sets, and note that this is a single set if $N=1$; since every vertex of $\mathcal{G}$ is passed twice by the bonds of $\mathcal{B}$, this set contains all the bonds of $\bar{\gamma}$.

The bonds of $\mathcal{B}$ will be in pairs of either the form $(v-1,v+2),(v,v+1)$ or the form $(v-1,v+2),(v,v+3)$, except for two bonds $(v-1,v)$ and $(v,v+1)$ when $v$ is a repeated vertex in $\bar{\gamma}$. 
Then in order for $\bar{\gamma}$ to have length $l$, $l-2N$ bonds must appear in these two types of pairs, meaning there will be $(l-2N)/2$ pairs of bonds. 
The only way for a vertex to not appear in $\bar{\gamma}$ is when a pair of bonds of the form $(v-1,v+2),(v,v+3)$ appears, which passes the vertex $v+1$ twice without the vertex itself appearing. 
Since there are $N$ vertices that appear twice in $\bar{\gamma}$, with the remaining vertices of $\bar{\gamma}$ appearing once, there are a total of $l-N$ vertices that appear at least once, so the number of vertices that do not appear in the $\bar{\gamma}$ is $n-(l-N)=n+N-l$, and this is the number of pairs of bonds of the form $(v-1,v+2),(v,v+3)$ that must appear. 
Starting two walks from the repeated vertex $v_0$, the bond $(v_0,v_0+1)$ adds 1 to the walk sum; each pair of bonds adds 4, with the pairs of the form $(v-1,v+2),(v,v+3)$ adding 2 extra; each repeated vertex $v_i$ has the two bonds $(v_i-1,v_i)$ and $(v_i,v_i+1)$ associated with it, which add 2 more; and the last bond $(v_0-1,v_0)$ adds 1 more. 
The two walks, by the construction, reach another repeated vertex, and they have a combined walk sum of
\begin{equation}  
\alpha = 1+\frac{4(l-2N)}{2}+2(n+N-l)+2(N-1)+1 =  2n \ .
\end{equation}
Thus, the two walks each have origin vertex $v_0$ and walk sum $n$, so their shared terminal vertex is $v_0+n$, which is the same vertex as $v_0$. 
Thus, since the two walks use all the bonds of $\mathcal{B}$, $o(\mathcal{B}) = t(\mathcal{B})$; therefore, by Lemma \ref{thm:bondstoPsO}, at least one pseudo orbit can be constructed from the bonds of $\mathcal{B}$. 
The walk sum of the pseudo orbit is $2n$, and its length is $l$ since $\mathcal{B}$ has $l$ elements; thus, it is a possible choice of $\bar{\gamma}$.

Of the $(l-2N)/2$ pairs of bonds, any $n+N-l$ of them may be chosen to take the form $(v-1,v+2),(v,v+3)$ (with the rest taking the form $(v-1,v+2),(v,v+1)$), creating
\begin{equation}
\label{eq:firstcomb}
    \binom{\frac{l}{2}-N}{n-l+N}
\end{equation}
different possible arrangements. 
In addition, each of the repeated vertices $v_i$, other than $v_0$, is of the form $v_i=v+2$ or $v_i=v+3$ for some pair of bonds.  
However, no two repeated vertices can be the terminal vertices for the same pair of bonds, as no adjacent vertices by a bond of arc 1 can both be repeated.  
Moreover, this pair of bonds cannot be the last one selected, as $v_0$ is either $v+2$ or $v+3$ for some $v$. 
There are then $((l-2N)/2)-1$ possible locations for $N-1$ repeated vertices, which creates
\begin{equation}
\label{eq:secondcomb}
    \binom{\frac{l}{2}-N-1}{N-1}
\end{equation}
different possibilities. 
These choices determine all of the bonds in $\bar{\gamma}$, so the number of different ways to select bonds for a pseudo orbit from a starting repeated vertex $v_0$ is the product of \eqref{eq:firstcomb} and \eqref{eq:secondcomb}.

Once the bonds have been selected, by Lemma \ref{thm:bondstoPsO}, the number of different pseudo orbits that can be formed with a set of bonds, since there are $N$ repeated vertices, is $2^N$. 
The total number of pseudo orbits that can be constructed from a repeated vertex $v_0$ (that is, the total number of pseudo orbits in which a particular vertex $v_0$ is repeated) is then
\begin{equation}
    2^N \binom{\frac{l}{2}-N}{n-l+N} \binom{\frac{l}{2}-N-1}{N-1} \ .
\end{equation}
Since there are $n$ vertices of $\mathcal{G}$, and $N$ repeated vertices in $\bar{\gamma}$ that could be used as $v_0$, the result follows.
\end{proof}

\begin{corollary}
\label{thm:max2enc}
For a graph $C_n^+(1,3)$, for any primitive pseudo orbit with length $0<l\leq n$, $N>0$ 2-enounters of length zero, and no other self-intersections,
\begin{equation}
    N \leq \frac{3l-2n}{4} \leq \frac{l}{4} \leq \frac{n}{4} \ .
\end{equation}
\end{corollary}
\begin{proof}
Let $\mathcal{G}=C_n^+(1,3)$, with $n>3$.

By Proposition \ref{thm:numPsOsecondfamily2enc}, the number of pseudo orbits of length $0<l\leq n$ with $N>0$ 2-encounters of length zero on $\mathcal{G}$ is given by equation \eqref{eq:numPsOsecondfamily2enc}. 
If the combination $\binom{(l/2)-N}{n-l+N}=0$, then the whole equation is equal to zero, meaning there exists no such pseudo orbit. 
In order for the combination $\binom{(l/2)-N}{n-l+N}$ to be greater than zero, 
$l/2-N \geq n-l+N$, and so, $2N \leq l/2-n+l=(3l-2n)/2$. 
Thus, $N \leq (3l-2n)/4$. 
Since $l\leq n$, $(3l-2n)/4 \leq l/4 \leq n/4$.
\end{proof}

\section{Quantum Circulant Graphs}
\label{sec:quantum}

A quantum graph regards a graph as a network of wires, rather than abstract relations between vertices; a graph $\mathcal{G}$ becomes a \textit{metric graph} when we associate to each bond $b$ a positive length $L_b$, such that the bond can be identified with the interval $[0,L_b]$.
We will only consider finite bond lengths here.
A metric graph can be quantized in one of two ways; here we will use the approach of Tanner \cite{T00}.
See also \cite{BK13,GS06} for more details, as well as for examples of applications of quantum graphs.

A metric graph $\mathcal{G}$ becomes a \textit{quantum graph} when we assign to each vertex $v$ of the graph a unitary $D_v^{out}\times D_v^{in}$ \textit{vertex scattering matrix}; unitarity implies the incoming and outgoing degrees are equal.
A popular choice of vertex scattering matrix for 4-regular directed graphs with $D_v^{out} = D_v^{in} = 2$ \cite{HH20, HH21, T00}, which we will use here, is the $2\times 2$ Discrete Fourier Transform (DFT) matrix,
\begin{equation}
\label{eq:DFTmatrix}
    \sigma^{(v)} = \dfrac{1}{\sqrt{2}} 
    \begin{pmatrix}
    1 & 1 \\ 1 & -1
    \end{pmatrix} \ .
\end{equation}
Then we construct the $B\times B$ \textit{bond scattering matrix} $S$, which collects the elements of the vertex scattering matrices; if $b$ is incoming at $v$ and $b'$ is outgoing at $v$, then 
\begin{equation}
    S_{b',b} = \delta_{t(b),v}\ \delta_{o(b'),v}\
\sigma_{b',b}^{(v)} \ .
\end{equation}
Let the matrix $L$ be a diagonal matrix of the bond lengths of $\mathcal{G}$, and the \textit{spectrum of $\mathcal{G}$} is defined to be the values $k^2$ that satisfy the \textit{secular equation}
\begin{equation}
\label{eq:secularequation}
    \det(I-S\textrm{e}^{\textrm{i} kL})=0 \ ,
\end{equation}
where $I$ is the $B\times B$ identity matrix.

A related function is the graph's \textit{characteristic polynomial}, 
\begin{equation}
    F_{\zeta}(k) = 
    \det(\zeta I - S\textrm{e}^{\textrm{i} kL}) = 
    \sum_{l=0}^B a_l(k) \zeta^{B-l} \ .
\end{equation}
Note that $F_1(k)=0$ is the secular equation \eqref{eq:secularequation}.
The average value of the coefficients can be shown to be zero, except $\langle a_0 \rangle = \langle a_B \rangle = 1$ \cite{BHJ12}.
Thus we consider the variance of the polynomial's coefficients, which is shown in \cite{HH20,HH21} for 4-regular directed quantum graphs $\mathcal{G}$ with $\sigma^{(v)}$ the matrix \eqref{eq:DFTmatrix} at each vertex $v$ to be 
\begin{equation}
\label{eq:varianceeq}
    \langle |a_l|^2 \rangle = 
    \dfrac{1}{2^l}
    \left( |\mathcal{P}_0^l| + \sum_{N=1}^l
    2^N |\hat{\mathcal{P}}_N^l| \right) \ ,
\end{equation}
where $\mathcal{P}_0^l$ is the set of $\mathcal{G}$'s primitive pseudo orbits of length $l$ with no self-intersections, and $\hat{\mathcal{P}}_N^l$ is the set of $\mathcal{G}$'s primitive pseudo orbits of length $l$ with $N$ self-intersections, each of which is a 2-encounter of length zero.
Note that the length $l$ is a topological length, the number of bonds in the pseudo orbit as discussed in section \ref{sec:background}.  
Pseudo orbits on a quantum graph also have a metric length, the sum of the metric lengths of the bonds, but \eqref{eq:varianceeq} does not depend on this.

\begin{table}[!htb]
\caption{\label{tab:family1} For a circulant graph $C_n^+(1,2)$, the sizes of the sets of relevant primitive pseudo orbits and the resulting variance for the first half of the characteristic polynomial's non-zero coefficients for $l\geq 1$ are given. The last two columns give the numerical values of the variance and the error between our formula and the numerics.}
\centering
\ \\
\begin{tabular}{ c | c c c c c }
	\hline
	$n=5$  & 
	$l$ & $|\mathcal{P}_0^l|$ 
	& $\langle |a_l|^2 \rangle$ 
	& Numerics & Error \\
	\hline
	  & 3 & 5 & 5/8 & 0.624966041 & 0.000033959 \\
	  & 4 & 5 & 5/16 & 0.312357578 & 0.000142422 \\
	  & 5 & 2 & 1/16 & 0.062499321 & 0.000000679 \\
	\hline
	$n=6$  & 
	$l$ & $|\mathcal{P}_0^l|$ 
	& $\langle |a_l|^2 \rangle$ 
	& Numerics & Error \\
	\hline
	  & 3 & 2 & 1/4 & 0.249998003 & 0.000001997 \\
	  & 4 & 9 & 9/16 & 0.562438456 & 0.000061544 \\
	  & 5 & 6 & 3/16 & 0.187507642 & -0.000007642 \\
	  & 6 & 2 & 1/32 & 0.031249628 & 0.000000372 \\
	\hline
	$n=7$  & $l$ & $|\mathcal{P}_0^l|$ 
	& $\langle |a_l|^2 \rangle$ 
	& Numerics & Error \\
	\hline
	  & 4 & 7 & 7/16 & 0.437247358 & 0.000252642 \\
	  & 5 & 14 & 7/16 & 0.437196166 & 0.000303834 \\
	  & 6 & 7 & 7/64 & 0.109341058 & 0.000033942 \\
	  & 7 & 2 & 1/64 & 0.015624921 & 0.000000079 \\
	\hline
	$n=8$  & $l$ & $|\mathcal{P}_0^l|$ 
	& $\langle |a_l|^2 \rangle$ 
	& Numerics & Error \\
	\hline
	  & 4 & 2 & 1/8 & 0.125046705 & -0.000046705 \\
	  & 5 & 16 & 1/2 & 0.499830464 & 0.000169536 \\
	  & 6 & 20 & 5/16 & 0.312173883 & 0.000326117 \\
	  & 7 & 8 & 1/16 & 0.062501263 & -0.000001263 \\
	  & 8 & 2 & 1/128 & 0.007812419 & 0.000000081 \\
	\hline
	$n=9$ \hspace{0.075cm}  & $l$ & $|\mathcal{P}_0^l|$ 
	& $\langle |a_l|^2 \rangle$ 
	& Numerics & Error \\
	\hline
	& 5 & 9 & 9/32 & 0.281324004 & -0.000074004 \\
	& 6 & 30 & 15/32 & 0.469212961 & -0.000462961 \\
	& 7 & 27 & 27/128 & 0.211048843 & -0.000111343 \\
    & 8 & 9 & 9/256 & 0.035164771 & -0.000008521 \\
	& 9 & 2 & 1/256 & 0.003906231 & 0.000000019 \\
	\hline
	$n=10$  & $l$ & $|\mathcal{P}_0^l|$ 
	& $\langle |a_l|^2 \rangle$ 
	& Numerics & Error \\
	\hline
	  & 5 & 2 & 1/16 & 0.062502445 & -0.000002445 \\
	  & 6 & 25 & 25/64 & 0.390888426 & -0.000263426 \\
	  & 7 & 50 & 25/64 & 0.389006327 & 0.001618673 \\
	  & 8 & 35 & 35/256 & 0.136729477 & -0.000010727 \\
	  & 9 & 10 & 5/256 & 0.019531102 & 0.000000148 \\
	  & 10 & 2 & 1/512 & 0.001953158 & -0.000000033 \\
	\hline
\end{tabular}
\end{table}

\begin{table}[!htb]
\caption{\label{tab:family2} For a circulant graph $C_n^+(1,3)$, the sizes of the sets of relevant primitive pseudo orbits and the resulting variance for the first half of the characteristic polynomial's non-zero coefficients for $l\geq 1$ are given. The last two columns give the numerical values of the variance and the error between our formula and the numerics.}
\centering
\ \\
\begin{tabular}{ c | c c c c c c c }
	\hline
	$n=5$  & 
	$l$ & $|\mathcal{P}_0^l|$ 
	& $|\hat{\mathcal{P}}_1^l|$ 
	& $|\hat{\mathcal{P}}_2^l|$
	& $\langle |a_l|^2 \rangle$ 
	& Numerics & Error \\
	\hline
	  & 3 & 5 & 0 & 0 & 5/8 & 0.625009866 & -0.000009866 \\
	  & 4 & 5 & 0 & 0 & 5/16 & 0.312505801 & -0.000005801 \\
	  & 5 & 2 & 0 & 0 & 1/16 & 0.062499021 & 0.000000979 \\
	\hline
	$n=6$  & 
	$l$ & $|\mathcal{P}_0^l|$ 
	& $|\hat{\mathcal{P}}_1^l|$ 
	& $|\hat{\mathcal{P}}_2^l|$
	& $\langle |a_l|^2 \rangle$ 
	& Numerics & Error \\
	\hline
	  & 2 & 3 & 0 & 0 & 3/4 & 0.750001357 & -0.000001357 \\
	  & 4 & 9 & 0 & 0 & 9/16 & 0.562538720 & -0.000038720 \\
	  & 6 & 4 & 24 & 0 & 13/16 & 0.812525458 & -0.000025458 \\
	\hline
	$n=7$  & 
	$l$ & $|\mathcal{P}_0^l|$ 
	& $|\hat{\mathcal{P}}_1^l|$ 
	& $|\hat{\mathcal{P}}_2^l|$
	& $\langle |a_l|^2 \rangle$ 
	& Numerics & Error \\
	\hline
	  & 3 & 7 & 0 & 0 & 7/8 & 0.874687835 & 0.000312165 \\
	  & 5 & 7 & 0 & 0 & 7/32 & 0.218729857 & 0.000020143 \\
	  & 6 & 7 & 14 & 0 & 35/64 & 0.546614990 & 0.000260010 \\
	  & 7 & 2 & 0 & 0 & 1/64 & 0.015622444 & 0.000002556 \\
	\hline
	$n=8$  & 
	$l$ & $|\mathcal{P}_0^l|$ 
	& $|\hat{\mathcal{P}}_1^l|$ 
	& $|\hat{\mathcal{P}}_2^l|$
	& $\langle |a_l|^2 \rangle$ 
	& Numerics & Error \\
	\hline
	  & 4 & 12 & 0 & 0 & 3/4 & 0.750132369 & -0.000132369 \\
	  & 6 & 16 & 0 & 0 & 1/4 & 0.249954214 & 0.000045786 \\
	  & 8 & 4 & 48 & 16 & 41/64 & 0.640908891 & -0.000283891 \\
	\hline
	$n=9$  & 
	$l$ & $|\mathcal{P}_0^l|$ 
	& $|\hat{\mathcal{P}}_1^l|$ 
	& $|\hat{\mathcal{P}}_2^l|$
	& $\langle |a_l|^2 \rangle$ 
	& Numerics & Error \\
	\hline
	& 3 & 3 & 0 & 0 & 3/8 & 0.375049482 & -0.000494825 \\
	& 5 & 18 & 0 & 0 & 9/16 & 0.562589299 & -0.000089299 \\
	& 6 & 3 & 0 & 0 & 3/64 & 0.046881185 & -0.000006185 \\
	& 7 & 9 & 0 & 0 & 9/128 & 0.070318141 & -0.000005641 \\
    & 8 & 9 & 54 & 0 & 117/256 & 0.457146967 & -0.000115717 \\
	& 9 & 2 & 0 & 0 & 1/256 & 0.003906314 & -0.000000064 \\
	\hline
	$n=10$  & 
	$l$ & $|\mathcal{P}_0^l|$ 
	& $|\hat{\mathcal{P}}_1^l|$ 
	& $|\hat{\mathcal{P}}_2^l|$
	& $\langle |a_l|^2 \rangle$ 
	& Numerics & Error \\
	\hline
	  & 4 & 10 & 0 & 0 & 5/8 & 0.625081315 & -0.000081315 \\
	  & 6 & 25 & 0 & 0 & 25/64 & 0.389251207 & 0.001373793 \\
	  & 8 & 25 & 20 & 0 & 65/256 & 0.253723049 & 0.000183201 \\
      & 10 & 4 & 80 & 120 & 161/256 & 0.628475168 & 0.000431082 \\
	\hline
\end{tabular}
\end{table}

Using the formulas determined in section \ref{sec:PsO} for the numbers of primitive pseudo orbits of these types in two different circulant graph families, we compute $\langle |a_l|^2 \rangle$ for $l=1,\dots,n$. 
See Tables \ref{tab:family1} and \ref{tab:family2} for these computations, as well as numerical results and corresponding error terms.
The numerical values of $\langle |a_l|^2 \rangle$ for $F_{\zeta}(k)$ were computed in MATLAB for circulant graphs having random bond lengths in the interval $[0.9,1.1]$, with the variance averaged over 1.5 million values of $k$.
The coefficient $a_0$ is always equal to one for a characteristic polynomial; the formula yields the same, as the null pseudo orbit $\bar{0}$ of length zero certainly lacks self-intersection.
A Riemann-Siegel lookalike formula for the coefficients, $a_l = a_B \bar{a}_{B-l}$, yields that the variance of the coefficients is symmetric about $l=n$, so only the nonzero variances of the first half of the coefficients are shown in Tables \ref{tab:family1} and \ref{tab:family2}, excluding $\langle |a_0|^2 \rangle = 1$.

\section{Conclusions}
\label{sec:conclusions}

By looking at the properties of graphs $C_n^+(a_1,a_2)$ for $a_1\neq a_2 \in \mathbb{N}$, we developed a closed formula, Theorem \ref{thm:generalPO}, that counts the number of primitive periodic orbits of any positive length on these graphs. 
This allows one to count primitive periodic orbits on a graph without referring to the graph's adjacency matrix and without a recursive formula.

We then counted the pseudo orbits belonging to relevant classes (no self-intersection or $N>0$ 2-encounters of length zero) for two specific families of graphs, $C_n^+(1,2)$ and $C_n^+(1,3)$, and used the counting results and \eqref{eq:varianceeq} to compute the variance of the coefficients of the characteristic polynomial of the corresponding quantum graphs. 
The numerical results that do not rely on counting pseudo orbits agree with our combinatorial results.

A next step may be to generalize the counting results for certain classes of pseudo orbits, section \ref{sec:PsO}, to other families of similar circulant graphs. 
One may also extend the counting methods developed in this paper to the study of other directed circulant graphs or other $k$-regular graphs with $D_v^{in}=D_v^{out}$, although certain methods, including the use of walk sums, may be most applicable in the context of circulant graphs. 

%\section*{Acknowledgments} %%% ADD HERE YOUR THANKS AND ACKNOWLEDGMENTS; please do not change spelling
 
%The authors would like to thank Jon Harrison for his helpful comments on the manuscript.

%%%%%%%% AUTHORS' INFORMATION. 
{\footnotesize  
\medskip
\medskip
\vspace*{1mm}

\noindent {\it Lauren Engelthaler}\\  
University of Dallas\\
1845 E Northgate Dr \\
Irving, TX 75062 \\
E-mail: {\tt lengelthaler@udallas.edu}\\ \\  

\noindent {\it Isaac Hellerman}\\  
University of Dallas \\
1845 E Northgate Dr \\
Irving, TX 75062 \\
E-mail: {\tt ihellerman@udallas.edu}\\ \\

\noindent {\it Tori Hudgins}\\  
University of Dallas \\
1845 E Northgate Dr \\
Irving, TX 75062 \\
E-mail: {\tt vhudgins@udallas.edu} \\ \\   
%%%%%%%%%%%% Please do not remove or move the } sign below, do not remove blank line before it

}

\begin{comment}
%%%%% Please do not change or remove the lines below, they will be changed in copyediting
\vspace*{1mm}\noindent\footnotesize{\date{ {\bf Received}: April 31, 2017\;\;\;{\bf Accepted}: June 31, 2017}}\\
\vspace*{1mm}\noindent\footnotesize{\date{  {\bf Communicated by Some Editor}}}
\end{comment}

\newpage
\section*{Appendix A: Numbers of Primitive Periodic Orbits for \texorpdfstring{$C_n^+(a_1,a_2)$}{}}

The primitive periodic orbits were generated by strings of adjacent vertices (without repeating strings that are equivalent cyclically) and then the total was counted for each graph in MATLAB. 
For all of the data shown, it was confirmed  in Python that the formula in Theorem \ref{thm:generalPO} agreed for the appropriate values of $a_1$ and $a_2$. 
The primitive periodic orbits of length one are not displayed for any graph, as there are no loops on any of the graphs we consider.

\subsection*{$C_n^+(1,2)$ (1st family)}

\begin{center}
\begin{tabular}{ c c | c c c c c c c c c c c c c c }
& & \multicolumn{14}{c}{$l$} \\

& & \textbf{2} & \textbf{3} & \textbf{4} & \textbf{5} & \textbf{6} & \textbf{7} & \textbf{8} & \textbf{9} & \textbf{10} & \textbf{11} & \textbf{12} & \textbf{13} & \textbf{14} & \textbf{15} \\
\hline
\multirow{8}{1em}{$n$} 
& \textbf{3}  & 3 & 2 & 3 & 6 & 9 & 18 & 30 & 56 & 99 & 186 & 335 & 630 & 1161 & 2182\\ 
& \textbf{4}  & 2 & 4 & 1 & 8 & 8 & 16 & 35 & 52 & 98 & 192 & 325 & 640 & 1162 & 2164\\ 
& \textbf{5}  & 0 & 5 & 5 & 2 & 10 & 25 & 20 & 50 & 126 & 175 & 290 & 695 & 1195 & 2001\\ 
& \textbf{6}  & 0 & 2 & 9 & 6 & 1 & 18 & 48 & 56 & 51 & 186 & 459 & 630 & 849 & 2182\\ 
& \textbf{7}  & 0 & 0 & 7 & 14 & 7 & 2 & 21 & 98 & 140 & 112 & 161 & 700 & 1716 & 2380\\ 
& \textbf{8}  & 0 & 0 & 2 & 16 & 20 & 8 & 1 & 32 & 160 & 336 & 320 & 224 & 620 & 2672\\ 
& \textbf{9}  & 0 & 0 & 0 & 9 & 30 & 27 & 9 & 2 & 36 & 270 & 678 & 891 & 639 & 543\\ 
& \textbf{10}  & 0 & 0 & 0 & 2 & 25 & 50 & 35 & 10 & 1 & 50 & 400 & 1320 & 2120 & 2002
\end{tabular}
\end{center}

\subsection*{$C_n^+(1,3)$ (2nd family)}

\begin{center}
\begin{tabular}{ c c | c c c c c c c c c c c c c c } 
& & \multicolumn{14}{c}{$l$} \\

& & \textbf{2} & \textbf{3} & \textbf{4} & \textbf{5} & \textbf{6} & \textbf{7} & \textbf{8} & \textbf{9} & \textbf{10} & \textbf{11} & \textbf{12} & \textbf{13} & \textbf{14} & \textbf{15} \\
\hline
\multirow{8}{1em}{$n$} 
& \textbf{4}  & 4 & 0 & 6 & 0 & 20 & 0 & 60 & 0 & 204 & 0 & 670 & 0 & 2340 & 0\\ 
& \textbf{5}  & 0 & 5 & 5 & 2 & 10 & 25 & 20 & 50 & 126 & 175 & 290 & 695 & 1195 & 2001\\ 
& \textbf{6}  & 3 & 0 & 6 & 0 & 21 & 0 & 60 & 0 & 204 & 0 & 670 & 0 & 2340 & 0\\ 
& \textbf{7}  & 0 & 7 & 0 & 7 & 14 & 2 & 49 & 63 & 35 & 294 & 287 & 427 & 1716 & 1610\\ 
& \textbf{8}  & 0 & 0 & 12 & 0 & 16 & 0 & 66 & 0 & 192 & 0 & 692 & 0 & 2304 & 0\\ 
& \textbf{9}  & 0 & 3 & 0 & 18 & 0 & 9 & 63 & 1 & 180 & 135 & 165 & 1188 & 288 & 2997\\ 
& \textbf{10}  & 0 & 0 & 10 & 0 & 25 & 0 & 40 & 0 & 254 & 0 & 580 & 0 & 2415 & 0\\
& \textbf{11}  & 0 & 0 & 0 & 22 & 0 & 33 & 11 & 11 & 220 & 2 & 726 & 242 & 770 & 3663
\end{tabular}
\end{center}

\subsection*{$C_n^+(1,4)$}

\begin{center}
\begin{tabular}{ c c | c c c c c c c c c c c c c c } 
& & \multicolumn{14}{c}{$l$} \\

& & \textbf{2} & \textbf{3} & \textbf{4} & \textbf{5} & \textbf{6} & \textbf{7} & \textbf{8} & \textbf{9} & \textbf{10} & \textbf{11} & \textbf{12} & \textbf{13} & \textbf{14} & \textbf{15} \\
\hline
\multirow{8}{1em}{$n$} 
& \textbf{5}  & 5 & 0 & 5 & 2 & 15 & 10 & 40 & 40 & 125 & 150 & 385 & 550 & 1285 & 2002\\ 
& \textbf{6}  & 0 & 8 & 0 & 0 & 28 & 0 & 0 & 168 & 0 & 0 & 1008 & 0 & 0 & 6552\\ 
& \textbf{7}  & 0 & 0 & 7 & 14 & 7 & 2 & 21 & 98 & 140 & 112 & 161 & 700 & 1716 & 2380\\ 
& \textbf{8}  & 4 & 0 & 0 & 8 & 0 & 40 & 1 & 112 & 32 & 240 & 330 & 448 & 1696 & 968\\ 
& \textbf{9}  & 0 & 9 & 0 & 0 & 27 & 0 & 0 & 167 & 0 & 0 & 1008 & 0 & 0 & 6552\\ 
& \textbf{10}  & 0 & 0 & 15 & 2 & 0 & 10 & 80 & 40 & 1 & 150 & 765 & 550 & 125 & 2002\\
& \textbf{11}  & 0 & 0 & 0 & 22 & 0 & 33 & 11 & 11 & 220 & 2 & 726 & 242 & 770 & 3663\\
& \textbf{12}  & 0 & 4 & 0 & 0 & 30 & 0 & 0 & 180 & 0 & 0 & 976 & 0 & 0 & 6604
\end{tabular}
\end{center}

\subsection*{$C_n^+(1,5)$}

\begin{center}
\begin{tabular}{ c c | c c c c c c c c c c c c c c } 
& & \multicolumn{14}{c}{$l$} \\

& & \textbf{2} & \textbf{3} & \textbf{4} & \textbf{5} & \textbf{6} & \textbf{7} & \textbf{8} & \textbf{9} & \textbf{10} & \textbf{11} & \textbf{12} & \textbf{13} & \textbf{14} & \textbf{15} \\
\hline
\multirow{8}{1em}{$n$} 
& \textbf{6}  & 6 & 0 & 6 & 0 & 20 & 0 & 60 & 0 & 204 & 0 & 670 & 0 & 2340 & 0\\ 
& \textbf{7}  & 0 & 7 & 0 & 7 & 14 & 2 & 49 & 63 & 35 & 294 & 287 & 427 & 1716 & 1610\\ 
& \textbf{8}  & 0 & 0 & 16 & 0 & 0 & 0 & 120 & 0 & 0 & 0 & 1360 & 0 & 0 & 0\\ 
& \textbf{9}  & 0 & 0 & 0 & 9 & 30 & 27 & 9 & 2 & 36 & 270 & 678 & 891 & 639 & 543\\ 
& \textbf{10}  & 5 & 0 & 0 & 0 & 10 & 0 & 70 & 0 & 253 & 0 & 710 & 0 & 2145 & 0\\
& \textbf{11}  & 0 & 11 & 0 & 0 & 22 & 11 & 0 & 99 & 132 & 2 & 440 & 1089 & 77 & 2200\\
& \textbf{12}  & 0 & 0 & 18 & 0 & 0 & 0 & 120 & 0 & 0 & 0 & 1360 & 0 & 0 & 0\\
& \textbf{13}  & 0 & 0 & 0 & 26 & 13 & 0 & 0 & 13 & 260 & 390 & 65 & 2 & 338 & 4329
\end{tabular}
\end{center}

\subsection*{$C_n^+(1,6)$}

\begin{center}
\begin{tabular}{ c c | c c c c c c c c c c c c c c } 
& & \multicolumn{14}{c}{$l$} \\

& & \textbf{2} & \textbf{3} & \textbf{4} & \textbf{5} & \textbf{6} & \textbf{7} & \textbf{8} & \textbf{9} & \textbf{10} & \textbf{11} & \textbf{12} & \textbf{13} & \textbf{14} & \textbf{15} \\
\hline
\multirow{8}{1em}{$n$} 
& \textbf{7}  & 7 & 0 & 7 & 0 & 21 & 2 & 56 & 14 & 175 & 70 & 525 & 308 & 1715 & 1274\\ 
& \textbf{8}  & 0 & 8 & 2 & 0 & 16 & 24 & 1 & 72 & 168 & 48 & 320 & 1056 & 612 & 1656\\ 
& \textbf{9}  & 0 & 3 & 9 & 0 & 0 & 45 & 27 & 1 & 108 & 378 & 162 & 207 & 1908 & 3003\\ 
& \textbf{10}  & 0 & 0 & 0 & 32 & 0 & 0 & 0 & 0 & 496 & 0 & 0 & 0 & 0 & 10912\\
& \textbf{11}  & 0 & 0 & 0 & 0 & 11 & 55 & 77 & 44 & 11 & 2 & 55 & 605 & 2332 & 4719\\
& \textbf{12}  & 6 & 0 & 0 & 0 & 0 & 12 & 0 & 112 & 0 & 504 & 1 & 1584 & 72 & 4004\\
& \textbf{13}  & 0 & 13 & 0 & 0 & 26 & 0 & 13 & 117 & 0 & 195 & 520 & 2 & 1859 & 2600\\
& \textbf{14}  & 0 & 0 & 21 & 0 & 0 & 2 & 112 & 14 & 0 & 70 & 1071 & 308 & 1 & 1274
\end{tabular}
\end{center}

\subsection*{$C_n^+(2,3)$}

\begin{center}
\begin{tabular}{ c c | c c c c c c c c c c c c c c } 
& & \multicolumn{14}{c}{$l$} \\

& & \textbf{2} & \textbf{3} & \textbf{4} & \textbf{5} & \textbf{6} & \textbf{7} & \textbf{8} & \textbf{9} & \textbf{10} & \textbf{11} & \textbf{12} & \textbf{13} & \textbf{14} & \textbf{15} \\
\hline
\multirow{8}{1em}{$n$} 
& \textbf{4}  & 2 & 4 & 1 & 8 & 8 & 16 & 35 & 52 & 98 & 192 & 325 & 640 & 1162 & 2164\\ 
& \textbf{5}  & 5 & 0 & 5 & 2 & 15 & 10 & 40 & 40 & 125 & 150 & 385 & 550 & 1285 & 2002\\ 
& \textbf{6}  & 3 & 2 & 0 & 12 & 0 & 30 & 21 & 56 & 120 & 120 & 462 & 462 & 1311 & 2180\\ 
& \textbf{7}  & 0 & 7 & 0 & 7 & 14 & 2 & 49 & 63 & 35 & 294 & 287 & 427 & 1716 & 1610\\
& \textbf{8}  & 0 & 8 & 2 & 0 & 16 & 24 & 1 & 72 & 168 & 48 & 320 & 1056 & 612 & 1656\\
& \textbf{9}  & 0 & 3 & 9 & 0 & 0 & 45 & 27 & 1 & 108 & 378 & 162 & 207 & 1908 & 3003\\
& \textbf{10}  & 0 & 0 & 15 & 2 & 0 & 10 & 80 & 40 & 1 & 150 & 765 & 550 & 125 & 2002\\
& \textbf{11}  & 0 & 0 & 11 & 11 & 0 & 0 & 33 & 154 & 44 & 2 & 198 & 1452 & 1573 & 341
\end{tabular}
\end{center}

\subsection*{$C_n^+(2,4)$}

\begin{center}
\begin{tabular}{ c c | c c c c c c c c c c c c c c } 
& & \multicolumn{14}{c}{$l$} \\

& & \textbf{2} & \textbf{3} & \textbf{4} & \textbf{5} & \textbf{6} & \textbf{7} & \textbf{8} & \textbf{9} & \textbf{10} & \textbf{11} & \textbf{12} & \textbf{13} & \textbf{14} & \textbf{15} \\
\hline
\multirow{8}{1em}{$n$} 
& \textbf{5}  & 0 & 5 & 5 & 2 & 10 & 25 & 20 & 50 & 126 & 175 & 290 & 695 & 1195 & 2001\\ 
& \textbf{6*}  & 6 & 4 & 6 & 12 & 18 & 36 & 60 & 112 & 198 & 372 & 670 & 1260 & 2322 & 4364\\ 
& \textbf{7}  & 0 & 0 & 7 & 14 & 7 & 2 & 21 & 98 & 140 & 112 & 161 & 700 & 1716 & 2380\\ 
& \textbf{8*}  & 4 & 8 & 2 & 16 & 16 & 32 & 70 & 104 & 196 & 384 & 650 & 1280 & 2324 & 4328\\
& \textbf{9}  & 0 & 0 & 0 & 9 & 30 & 27 & 9 & 2 & 36 & 270 & 678 & 891 & 639 & 543\\
& \textbf{10*}  & 0 & 10 & 10 & 4 & 20 & 50 & 40 & 100 & 252 & 350 & 580 & 1390 & 2390 & 4002\\
& \textbf{11}  & 0 & 0 & 0 & 0 & 11 & 55 & 77 & 44 & 11 & 2 & 55 & 605 & 2332 & 4719\\
& \textbf{12*}  & 0 & 4 & 18 & 12 & 2 & 36 & 96 & 112 & 102 & 372 & 918 & 1260 & 1698 & 4364
\end{tabular}
\end{center}
*not a connected graph

\subsection*{$C_n^+(2,5)$}

\begin{center}
\begin{tabular}{ c c | c c c c c c c c c c c c c c } 
& & \multicolumn{14}{c}{$l$} \\

& & \textbf{2} & \textbf{3} & \textbf{4} & \textbf{5} & \textbf{6} & \textbf{7} & \textbf{8} & \textbf{9} & \textbf{10} & \textbf{11} & \textbf{12} & \textbf{13} & \textbf{14} & \textbf{15} \\
\hline
\multirow{8}{1em}{$n$} 
& \textbf{6}  & 0 & 8 & 0 & 0 & 28 & 0 & 0 & 168 & 0 & 0 & 1008 & 0 & 0 & 6552\\ 
& \textbf{7}  & 7 & 0 & 7 & 0 & 21 & 2 & 56 & 14 & 175 & 70 & 525 & 308 & 1715 & 1274\\ 
& \textbf{8}  & 0 & 0 & 2 & 16 & 20 & 8 & 1 & 32 & 160 & 336 & 320 & 224 & 620 & 2672\\ 
& \textbf{9}  & 0 & 9 & 0 & 0 & 27 & 0 & 0 & 167 & 0 & 0 & 1008 & 0 & 0 & 6552\\
& \textbf{10}  & 5 & 0 & 0 & 2 & 0 & 30 & 0 & 140 & 0 & 420 & 55 & 990 & 700 & 2002\\
& \textbf{11}  & 0 & 0 & 11 & 11 & 0 & 0 & 33 & 154 & 44 & 2 & 198 & 1452 & 1573 & 341\\
& \textbf{12}  & 0 & 12 & 0 & 0 & 26 & 0 & 0 & 156 & 0 & 0 & 976 & 0 & 0 & 6500\\
& \textbf{13}  & 0 & 0 & 0 & 13 & 0 & 65 & 0 & 52 & 52 & 13 & 858 & 2 & 2756 & 390
\end{tabular}
\end{center}

\subsection*{$C_n^+(3,4)$}

\begin{center}
\begin{tabular}{ c c | c c c c c c c c c c c c c c } 
& & \multicolumn{14}{c}{$l$} \\

& & \textbf{2} & \textbf{3} & \textbf{4} & \textbf{5} & \textbf{6} & \textbf{7} & \textbf{8} & \textbf{9} & \textbf{10} & \textbf{11} & \textbf{12} & \textbf{13} & \textbf{14} & \textbf{15} \\
\hline
\multirow{8}{1em}{$n$} 
& \textbf{5}  & 0 & 5 & 5 & 2 & 10 & 25 & 20 & 50 & 126 & 175 & 290 & 695 & 1195 & 2001\\ 
& \textbf{6}  & 3 & 2 & 0 & 12 & 0 & 30 & 21 & 56 & 120 & 120 & 462 & 462 & 1311 & 2180\\ 
& \textbf{7}  & 7 & 0 & 7 & 0 & 21 & 2 & 56 & 14 & 175 & 70 & 525 & 308 & 1715 & 1274\\ 
& \textbf{8}  & 4 & 0 & 0 & 8 & 0 & 40 & 1 & 112 & 32 & 240 & 330 & 448 & 1696 & 968\\
& \textbf{9}  & 0 & 3 & 0 & 18 & 0 & 9 & 63 & 1 & 180 & 135 & 165 & 1188 & 288 & 2997\\
& \textbf{10}  & 0 & 10 & 0 & 2 & 20 & 0 & 35 & 90 & 1 & 300 & 400 & 70 & 2145 & 2000\\
& \textbf{11}  & 0 & 11 & 0 & 0 & 22 & 11 & 0 & 99 & 132 & 2 & 440 & 1089 & 77 & 2200\\
& \textbf{12}  & 0 & 4 & 3 & 0 & 0 & 60 & 0 & 0 & 252 & 180 & 0 & 660 & 2544 & 364
\end{tabular}
\end{center}

\subsection*{$C_n^+(3,5)$}

\begin{center}
\begin{tabular}{ c c | c c c c c c c c c c c c c c } 
& & \multicolumn{14}{c}{$l$} \\

& & \textbf{2} & \textbf{3} & \textbf{4} & \textbf{5} & \textbf{6} & \textbf{7} & \textbf{8} & \textbf{9} & \textbf{10} & \textbf{11} & \textbf{12} & \textbf{13} & \textbf{14} & \textbf{15} \\
\hline
\multirow{8}{1em}{$n$} 
& \textbf{6}  & 3 & 0 & 6 & 0 & 21 & 0 & 60 & 0 & 204 & 0 & 670 & 0 & 2340 & 0\\ 
& \textbf{7}  & 0 & 0 & 7 & 14 & 7 & 2 & 21 & 98 & 140 & 112 & 161 & 700 & 1716 & 2380\\ 
& \textbf{8}  & 8 & 0 & 8 & 0 & 24 & 0 & 66 & 0 & 216 & 0 & 688 & 0 & 2376 & 0\\ 
& \textbf{9}  & 0 & 3 & 9 & 0 & 0 & 45 & 27 & 1 & 108 & 378 & 162 & 207 & 1908 & 3003\\
& \textbf{10}  & 5 & 0 & 0 & 0 & 10 & 0 & 70 & 0 & 253 & 0 & 710 & 0 & 2145 & 0\\
& \textbf{11}  & 0 & 11 & 0 & 0 & 22 & 11 & 0 & 99 & 132 & 2 & 440 & 1089 & 77 & 2200\\
& \textbf{12}  & 0 & 0 & 3 & 0 & 40 & 0 & 42 & 0 & 156 & 0 & 905 & 0 & 2028 & 0\\
& \textbf{13}  & 0 & 13 & 0 & 0 & 26 & 0 & 13 & 117 & 0 & 195 & 520 & 2 & 1859 & 2600
\end{tabular}
\end{center}

\newpage
\section*{Appendix B: Primitive Pseudo Orbits for  \texorpdfstring{$C_n^+(1,3)$}{}}

These lists include all of the primitive pseudo orbits of length $0<l\leq n$ for the second family of graphs, $C_n^+(1,3)$ for $5\leq n\leq 10$, sorted by self-intersections (no self-intersections, $N>0$ 2-encounters of length zero, and self-intersections of length greater than zero). 
The periodic orbits were generated using code written in MATLAB that constructed all possible circuits as a sequence of vertices on a graph (without repeating strings that are equivalent cyclically); each sequence wraps around so that the terminal vertex is the origin vertex. 
Combining these into pseudo orbits and sorting was primarily done by hand. 
Parentheses denote a new primitive pseudo orbit, while commas separate primitive periodic orbits contained in a pseudo orbit.

\setlength{\parindent}{0pt}
\setlength{\parskip}{0.5em}

\subsection*{Primitive pseudo orbits for $n=5$, $l\leq 5$, sorted by self-intersections:}

$l=3$

\textit{No self-intersection} (5 orbits):
(012), 
(014), 
(034), 
(123), 
(234)

$l=4$

\textit{No self-intersection} (5 orbits):
(0142), 
(0312), 
(0314), 
(0342), 
(1423)

$l=5$

\textit{No self-intersection} (2 orbits):
(01234), 
(03142)

\subsection*{Primitive pseudo orbits for $n=6$, $l\leq 6$, sorted by self-intersections:}

$l=2$

\textit{No self-intersection} (3 orbits):
(03),
(14),
(25)

$l=4$

\textit{No self-intersection} (9 orbits):
(0123),
(0125),
(0145),
(0345),
(1234),
(2345),
(03, 14),
(03, 25),\\
(14, 25)

$l=6$

\textit{No self-intersection} (4 orbits):
(012345),
(014523),
(034125),
(03, 14, 25)

\textit{One 2-encounter of length 0} (24 orbits): 
(012345),
(012503),
(012523),
(014123),
(014125),
(014503),
(014523),
(014525),
(034123),
(034125),
(034145),
(034523),
(034525),
(125234),
(145234),
(03, 0125),
(03, 0145),
(03, 1234),
(03, 2345),
(14, 0123),
(14, 0125),
(14, 0345),
(14, 2345),
(25, 0123),
(25, 0145),
(25, 0345),
(25, 1234)

\pagebreak
\textit{At least one 2-encounter of length 1 or greater} (12 orbits):
(012303),
(012525),
(014145),
(030345),
(123414),
(234525),
(03, 0123),
(03, 0345),
(14, 0145),
(14, 1234),
(25, 0125),
(25, 2345)

\subsection*{Primitive pseudo orbits for $n=7$, $l\leq 7$, sorted by self-intersections:}

$l=3$

\textit{No self-intersection} (7 orbits):
(014),
(034),
(036),
(125),
(145),
(236),
(256)

$l=5$

\textit{No self-intersection} (7 orbits):
(01234),
(01236),
(01256),
(01456),
(03456),
(12345),
(23456)

$l=6$

\textit{No self-intersection} (7 orbits):
(014, 236),
(014, 256),
(034, 125),
(034, 256),
(036, 125),
(036, 145),\\
(145, 236)

\textit{One 2-encounter of length 0} (14 orbits):
(012514),
(014036),
(034514),
(036234),
(036256),
(123625),
(145625),
(014, 036),
(014, 125),
(034, 145),
(034, 236),
(036, 256),
(125, 236),
(145, 256)

\textit{At least one 2-encounter of length 1 or greater} (14 orbits):
(014034),
(014514),
(034036),
(036236),
(125145),
(125625),
(236256),
(014, 034),
(014, 145),
(034, 036),
(036, 236),
(125, 145),
(125, 256),\\
(236, 256)

$l=7$

\textit{No self-intersection} (2 orbits):
(0123456),
(0362514),

\subsection*{Primitive pseudo orbits for $n=8$, $l\leq8$, sorted by self-intersections:}

$l=4$

\textit{No self-intersection} (12 orbits):
(0125),
(0145),
(0147),
(0345),
(0347),
(0367),
(1236),
(1256),
(1456),
(2347),
(2367),
(2567)

$l=6$

\textit{No self-intersection} (16 orbits):
(012345),
(012347),
(012367),
(012567),
(014567),
(014725),
(034567),
(034725),
(036125),
(036145),
(036147),
(036725),
(123456),
(147236),
(147256),
(234567)

\pagebreak

$l=8$

\textit{No self-intersection} (4 orbits):
(01234567),
(03614725),
(0145, 2367),
(0347, 1256)

\textit{One 2-encounter of length 0} (48 orbits):
(01234725),
(01236145),
(01236147),
(01236725),
(01250347),
(01250367),
(01256147),
(01450367),
(01456725),
(01472345),
(01472367),
(01472567),
(03456125),
(03456147),
(03456725),
(03472567),
(03612345),
(03612347),
(03612567),
(03614567),
(03672345),
(12347256),
(14567236),
(14723456),
(0125, 0347),
(0125, 0367),
(0125, 2347),
(0125, 2367),
(0145, 0367),
(0145, 1236),
(0145, 2347),\\
(0145, 2567),
(0147, 1236),
(0147, 1256),
(0147, 2367),
(0147, 2567),
(0345, 1236),\\
(0345, 1256),
(0345, 2367),
(0345, 2567),
(0347, 1236),
(0347, 1456),
(0347, 2567),\\
(0367, 1256),
(0367, 1456),
(1256, 2347),
(1456, 2347),
(1456, 2367)

\textit{Two 2-encounters of length 0} (16 orbits):
(01256145),
(01450347),
(01456125),
(01470345),
(03472367),
(03672347),
(12367256),
(12567236),
(0125, 1456),
(0145, 0347),
(0145, 1256),
(0147, 0345),
(0347, 2367),
(0367, 2347),
(1236, 2567),
(1256, 2367)

\textit{At least one 2-encounter of length 1 or greater} (64 orbits):
(01236125),
(01250145),
(01250147),
(01250345),
(01256125),
(01256725),
(01450147),
(01450345),
(01456145),
(01456147),
(01470347),
(01470367),
(01472347),
(03450347),
(03450367),
(03456145),
(03470367),
(03472345),
(03472347),
(03612367),
(03672367),
(03672567),
(12347236),
(12361256),
(12361456),
(12367236),
(12561456),
(12567256),
(14567256),
(23472367),
(23472567),
(23672567),
(0125, 0145),
(0125, 0147),
(0125, 0345),
(0125, 1236),\\
(0125, 1256),
(0125, 2567),
(0145, 0147),
(0145, 0345),
(0145, 1456),
(0147, 0347),\\
(0147, 0367),
(0147, 1456),
(0147, 2347),
(0345, 0347),
(0345, 0367),
(0345, 1456),\\
(0345, 2347),
(0347, 0367),
(0347, 2347),
(0367, 1236),
(0367, 2367),
(0367, 2567),\\
(1236, 1256),
(1236, 1456),
(1236, 2347),
(1236, 2367),
(1256, 1456),
(1256, 2567),\\
(1456, 2567),
(2347, 2367),
(2347, 2567),
(2367, 2567)

\subsection*{Primitive pseudo orbits for $n=9$, $l\leq 9$, sorted by self-intersections:}

$l=3$

\textit{No self-intersection} (3 orbits):
(036), (147), (258)

$l=5$

\textit{No self-intersection} (18 orbits):
(01236), (01256), (01258), (01456), (01458), (01478), (03456), (03458), (03478), (03678), (12347), (12367), (12567), (14567), (23458), (23478), (23678), (25678)

$l=6$

\textit{No self-intersection} (3 orbits):
(036, 147), (036, 258), (147, 258)

\pagebreak
$l=7$

\textit{No self-intersection} (9 orbits):
(0123456), (0123458), (0123478), (0123678), (0125678), (0145678), 
(0345678), (1234567), (2345678)

$l=8$

\textit{No self-intersection} (9 orbits):
(01458236),
(01478236),
(01478256),
(03471256),
(03471258),
(03478256),
(03671258),
(03671458),
(14582367)

\textit{One 2-encounter of length 0} (54 orbits):
(01258036),
(01258236),
(01458036),
(01458256),
(01471236),
(01471256),
(01471258),
(01478036),
(01478258),
(03458236),
(03458256),
(03471236),
(03471456),
(03471458),
(03478236),
(03478258),
(03671256),
(03671456),
(03671478),
(03678256),
(03678258),
(12582347),
(12582367),
(14582347),
(14582567),
(14782367),
(14782567),
(036, 01258),
(036, 01458),
(036, 01478),
(036, 12347),\\
(036, 12567),
(036, 14567),
(036, 23458),
(036, 23478),
(036, 25678),
(147, 01236),\\
(147, 01256),
(147, 01258),
(147, 03456),
(147, 03458),
(147, 03678),
(147, 23458),\\
(147, 23678),
(147, 25678),
(258, 01236),
(258, 01456),
(258, 01478),
(258, 03456),\\
(258, 03478),
(258, 03678),
(258, 12347),
(258, 12367),
(258, 14567)

\textit{At least one 2-encounter of length 1 or greater} (54 orbits):
(01236036),
(01256036),
(01258256),
(01258258),
(01456036),
(01458258),
(01471456),
(01471458),
(01471478),
(03456036),
(03458036),
(03458258),
(03471478),
(03478036),
(03671236),
(03678036),
(03678236),
(12347147),
(12367147),
(12567147),
(12582567),
(14567147),
(14782347),
(23458258),
(23478258),
(23678258),
(25678258),
(036, 01236),
(036, 01256),
(036, 01456),
(036, 03456),
(036, 03458),
(036, 03478),
(036, 03678),
(036, 12367),
(036, 23678),\\
(147, 01456),
(147, 01458),
(147, 01478),
(147, 03478),
(147, 12347),
(147, 12367),\\
(147, 12567),
(147, 14567),
(147, 23478),
(258, 01256),
(258, 01258),
(258, 01458),\\
(258, 03458),
(258, 12567),
(258, 23458),
(258, 23478),
(258, 23678),
(258, 25678)

$l=9$

\textit{No self-intersection} (2 orbits):
(012345678),
(036, 147, 258)

\subsection*{Primitive pseudo orbits for $n=10$, $l\leq 10$, sorted by self-intersections:}

$l=4$

\textit{No self-intersection} (10 orbits):
(0147),
(0347),
(0367),
(0369),
(1258),
(1458),
(1478),
(2369),
(2569),
(2589)

$l=6$

\textit{No self-intersection} (25 orbits):
(012347),
(012367),
(012369),
(012567),
(012569),
(012589),
(014567),
(014569),
(014589),
(014789),
(034567),
(034569),
(034589),
(034789),
(036789),
(123458),
(123478),
(123678),
(125678),
(145678),
(234569),
(234589),
(234789),
(236789),
(256789)

\newpage
$l=8$

\textit{No self-intersection} (25 orbits):
(01234567),
(01234569),
(01234589),
(01234789),\\
(01236789),
(01256789),
(01456789),
(03456789),
(12345678),
(23456789),
(0147, 2369),
(0147, 2569),
(0147, 2589),
(0347, 1258),
(0347, 2569),
(0347, 2589),
(0367, 1258),\\
(0367, 1458),
(0367, 2589),
(0369, 1258),
(0369, 1458),
(0369, 1478),
(1458, 2369),\\
(1478, 2369),
(1478, 2569)

\textit{One 2-encounter of length 0} (20 orbits):
(01258147),
(03458147),
(03678147),
(03690147),
(03692347),
(03692567),
(03692589),
(12369258),
(14569258),
(14789258),
(0147, 0369),
(0147, 1258),
(0347, 1458),
(0347, 2369),
(0367, 1478),
(0367, 2569),
(0369, 2589),\\
(1258, 2369),
(1458, 2569),
(1478, 2589)

\textit{At least one 2-encounter of length 1 or greater} (40 orbits):

(01458147),
(01470347),
(01470367),
(01478147),
(03470367),
(03470369),
(03478147),
(03670369),
(03692367),
(03692369),
(03692569),
(12569258),
(12581458),
(12581478),
(12589258),
(14581478),
(14589258),
(23692569),
(23692589),
(25692589),
(0147, 0347),
(0147, 0367),
(0147, 1458),
(0147, 1478),
(0347, 0367),
(0347, 0369),
(0347, 1478),\\
(0367, 0369),
(0367, 2369),
(0369, 2369),
(0369, 2569),
(1258, 1458),
(1258, 1478),\\
(1258, 2569),
(1258, 2589),
(1458, 1478),
(1458, 2589),
(2369, 2569),
(2369, 2589),\\
(2569, 2589)

$l=10$

\textit{No self-intersection} (4 orbits):
(0123456789), (0145892367), (0347812569), (0369258147)

\textit{One 2-encounter of length 0} (80 orbits):
(0125690347),
(0125692347),
(0125890347),
(0125890367),
(0125892347),
(0125892367),
(0145692347),
(0145692367),
(0145812367),
(0145812369),
(0145890367),
(0145892347),
(0145892369),
(0145892567),
(0147812369),
(0147812569),
(0147892367),
(0147892369),
(0147892567),
(0147892569),
(0345812367),
(0345812369),
(0345812567),
(0345812569),
(0345892367),
(0345892567),
(0347812369),
(0347812567),
(0347812589),
(0347814569),
(0347892567),
(0347892569),
(0367812569),
(0367812589),
(0367814569),
(0367814589),
(1256923478),
(1456923478),
(1456923678),
(1458923678),
(0147, 234569),
(0147, 234589),
(0147, 236789),
(0147, 256789),\\
(0347, 012569),
(0347, 012589),
(0347, 125678),
(0347, 256789),
(0367, 012589),\\
(0367, 014589),
(0367, 123458),
(0367, 234589),
(0369, 123458),
(0369, 123478),\\
(0369, 125678),
(0369, 145678),
(1258, 034567),
(1258, 034569),
(1258, 034789),\\
(1258, 036789),
(1458, 012367),
(1458, 012369),
(1458, 036789),
(1458, 236789),\\
(1478, 012369),
(1478, 012569),
(1478, 034569),
(1478, 234569),
(2369, 014567),\\
(2369, 014589),
(2369, 014789),
(2369, 145678),
(2569, 012347),
(2569, 014789),\\
(2569, 034789),
(2569, 123478),
(2589, 012347),
(2589, 012367),
(2589, 014567),\\
(2589, 034567)

\textit{Two 2-encounters of length 0} (120 orbits):
(0123458147),
(0123470369),
(0123678147),
(0123690347),
(0123692567),
(0123692589),
(0125670369),
(0125678147),
(0125690367),
(0125692367),
(0125814567),
(0125814569),
(0125814789),
(0125892369),
(0145670369),
(0145690347),
(0145690367),
(0145692589),
(0145812347),
(0145812567),
(0145812569),
(0145890347),
(0145892569),
(0147034569),
(0147034589),
(0147036789),
(0147812367),
(0147812567),
(0147812589),
(0147890367),
(0345678147),
(0345692367),
(0345692589),
(0345814789),
(0345892369),
(0345892569),
(0347812367),
(0347814567),
(0347814589),
(0347892367),
(0347892369),
(0367812347),
(0367892347),
(0367892569),
(0369234567),
(0369234589),
(0369234789),
(0369256789),
(1234569258),
(1234789258),
(1236789258),
(1236925678),
(1256923458),
(1256923678),
(1258923478),
(1258923678),
(1456789258),
(1458923478),
(1458925678),
(1478923458),
(0147, 034569),
(0147, 034589),
(0147, 036789),
(0147, 123458),
(0147, 123678),
(0147, 125678),
(0347, 012369),
(0347, 014569),\\
(0347, 014589),
(0347, 123678),
(0347, 145678),
(0347, 236789),
(0367, 012569),\\
(0367, 014569),
(0367, 014789),
(0367, 123478),
(0367, 234569),
(0367, 234789),\\
(0369, 012347),
(0369, 012567),
(0369, 014567),
(0369, 234589),
(0369, 234789),\\
(0369, 256789),
(1258, 014567),
(1258, 014569),
(1258, 014789),
(1258, 234569),\\
(1258, 234789),
(1258, 236789),
(1458, 012347),
(1458, 012567),
(1458, 012569),\\
(1458, 034789),
(1458, 234789),
(1458, 256789),
(1478, 012367),
(1478, 012567),\\
(1478, 012589),
(1478, 034567),
(1478, 034589),
(1478, 234589),
(2369, 012567),\\
(2369, 012589),
(2369, 034567),
(2369, 034589),
(2369, 034789),
(2369, 125678),\\
(2569, 012367),
(2569, 014589),
(2569, 034589),
(2569, 036789),
(2569, 123458),\\
(2569, 123678),
(2589, 012369),
(2589, 014569),
(2589, 034569),
(2589, 123478),\\
(2589, 123678),
(2589, 145678)

\textit{At least one 2-encounter of length 1 or greater} (300 orbits):
(0123470147),
(0123470347),
(0123470367),
(0123478147),
(0123670147),
(0123670347),
(0123670367),
(0123670369),
(0123690147),
(0123690367),
(0123690369),
(0123692347),
(0123692367),
(0123692369),
(0123692569),
(0125670147),
(0125670347),
(0125670367),
(0125690147),
(0125690369),
(0125692369),
(0125692567),
(0125692569),
(0125692589),
(0125812347),
(0125812367),
(0125812369),
(0125812567),
(0125812569),
(0125812589),
(0125814589),
(0125890147),
(0125890369),
(0125892567),
(0125892569),
(0125892589),
(0145670147),
(0145670347),
(0145670367),
(0145678147),
(0145690147),
(0145690369),
(0145692369),
(0145692567),
(0145692569),
(0145812589),
(0145814567),
(0145814569),
(0145814589),
(0145814789),
(0145890147),
(0145890369),
(0145892589),
(0147014789),
(0147034567),
(0147034789),
(0147812347),
(0147814567),
(0147814569),
(0147814589),
(0147814789),
(0147890347),
(0147890369),
(0147892347),
(0147892589),
(0345670347),
(0345670367),
(0345670369),
(0345690347),
(0345690367),
(0345690369),
(0345692347),
(0345692369),
(0345692567),
(0345692569),
(0345812347),
(0345812589),
(0345814567),
(0345814569),
(0345814589),
(0345890347),
(0345890367),
(0345890369),
(0345892347),
(0345892589),
(0347034789),
(0347036789),
(0347812347),
(0347814789),
(0347890367),
(0347890369),
(0347892347),
(0347892589),
(0367036789),
(0367812367),
(0367812369),
(0367812567),
(0367814567),
(0367814789),
(0367890369),
(0367892367),
(0367892369),
(0367892567),
(0367892589),
(0369234569),
(0369236789),
(1234581258),
(1234581458),
(1234581478),
(1234589258),
(1234781258),
(1234781458),
(1234781478),
(1236781258),
(1236781458),
(1236781478),
(1236923458),
(1236923478),
(1236923678),
(1256781258),
(1256781458),
(1256781478),
(1256789258),
(1256925678),
(1258145678),
(1258923458),
(1258925678),
(1456781458),
(1456781478),
(1456923458),
(1456925678),
(1458923458),
(1478923478),
(1478923678),
(1478925678),
(2345692369),
(2345692569),
(2345692589),
(2345892369),
(2345892569),
(2345892589),
(2347892369),
(2347892569),
(2347892589),
(2367892369),
(2367892569),
(2367892589),
(2369256789),
(2567892569),
(2567892589),
(0147, 012347),
(0147, 012367),
(0147, 012369),
(0147, 012567),
(0147, 012569),
(0147, 012589),
(0147, 014567),\\
(0147, 014569),
(0147, 014589),
(0147, 014789),
(0147, 034567),
(0147, 034789),\\
(0147, 123478),
(0147, 145678),
(0147, 234789),
(0347, 012347),
(0347, 012367),\\
(0347, 012567),
(0347, 014567),
(0347, 014789),
(0347, 034567),
(0347, 034569),\\
(0347, 034589),
(0347, 034789),
(0347, 036789),
(0347, 123458),
(0347, 123478),\\
(0347, 234569),
(0347, 234589),
(0347, 234789),
(0367, 012347),
(0367, 012367),\\
(0367, 012369),
(0367, 012567),
(0367, 014567),
(0367, 034567),
(0367, 034569),\\
(0367, 034589),
(0367, 034789),
(0367, 036789),
(0367, 123678),
(0367, 125678),\\
(0367, 145678),
(0367, 236789),
(0367, 256789),
(0369, 012367),
(0369, 012369),\\
(0369, 012569),
(0369, 012589),
(0369, 014569),
(0369, 014589),
(0369, 014789),\\
(0369, 034567),
(0369, 034569),
(0369, 034589),
(0369, 034789),
(0369, 036789),\\
(0369, 123678),
(0369, 234569),
(0369, 236789),
(1258, 012347),
(1258, 012367),\\
(1258, 012369),
(1258, 012567),
(1258, 012569),
(1258, 012589),
(1258, 014589),\\
(1258, 034589),
(1258, 123458),
(1258, 123478),
(1258, 123678),
(1258, 125678),\\
(1258, 145678),
(1258, 234589),
(1258, 256789),
(1458, 012589),
(1458, 014567),\\
(1458, 014569),
(1458, 014589),
(1458, 014789),
(1458, 034567),
(1458, 034569),\\
(1458, 034589),
(1458, 123458),
(1458, 123478),
(1458, 123678),
(1458, 125678),\\
(1458, 145678),
(1458, 234569),
(1458, 234589),
(1478, 012347),
(1478, 014567),\\
(1478, 014569),
(1478, 014589),
(1478, 014789),
(1478, 034789),
(1478, 036789),\\
(1478, 123458),
(1478, 123478),
(1478, 123678),
(1478, 125678),
(1478, 145678),\\
(1478, 234789),
(1478, 236789),
(1478, 256789),
(2369, 012347),
(2369, 012367),\\
(2369, 012369),
(2369, 012569),
(2369, 014569),
(2369, 034569),
(2369, 036789),\\
(2369, 123458),
(2369, 123478),
(2369, 123678),
(2369, 234569),
(2369, 234589),\\
(2369, 234789),
(2369, 236789),
(2369, 256789),
(2569, 012369),
(2569, 012567),\\
(2569, 012569),
(2569, 012589),
(2569, 014567),
(2569, 014569),
(2569, 034567),\\
(2569, 034569),
(2569, 125678),
(2569, 145678),
(2569, 234569),
(2569, 234589),\\
(2569, 234789),
(2569, 236789),
(2569, 256789),
(2589, 012567),
(2589, 012569),\\
(2589, 012589),
(2589, 014589),
(2589, 014789),
(2589, 034589),
(2589, 034789),\\
(2589, 036789),
(2589, 123458),
(2589, 125678),
(2589, 234569),
(2589, 234589),\\
(2589, 234789),
(2589, 236789),
(2589, 256789)

\end{document}